\documentclass[11pt, reqno]{amsart}
\usepackage{amssymb, hyperref}
\usepackage{amsthm}
\usepackage{amsmath}
\usepackage{graphicx}
\usepackage[all,2cell]{xy}
\usepackage{verbatim}
\usepackage{tikz}
\usepackage{tikz-cd}
\usepackage{hyperref}
\usepackage{array}
\usepackage[capitalise]{cleveref}
\usepackage{epsfig,graphicx}
\usepackage{tikz}
\usepackage{tikz-cd}
\usepackage{float}
\usepackage{capt-of}
\tikzset{node distance=2cm, auto}
\usepackage[vcentermath]{youngtab}
\usepackage{ytableau}
\usepackage{etoolbox}
\usetikzlibrary{graphs}
\usepackage{listings}
\usetikzlibrary{backgrounds}
\usepackage{multicol}
\usetikzlibrary{arrows,decorations.markings}
\tikzstyle{vertex}=[circle, draw, inner sep=0pt, minimum size=6pt]

\usetikzlibrary{matrix,arrows,decorations.pathmorphing}
\usepackage{mathrsfs}
\usepackage{caption}
\numberwithin{equation}{section}
\newtheorem*{theorem*}{Theorem}
\newtheorem*{corollary*}{\bf Corollary}
\newtheorem*{remark*}{\bf Remark}
\usepackage{lipsum}
\usepackage{lineno}
\newtheorem{theorem}{Theorem}[section]
\newtheorem{corollary}[theorem]{Corollary}
\newtheorem{definition}[theorem]{Definition}
\newtheorem{example}[theorem]{Example}
\newtheorem{lemma}[theorem]{Lemma}

\newtheorem{remark}[theorem]{Remark}

\newcommand{\eat}[1]{}

\title[Torus quotients of Richardson varieties in $G_{r,qr+1}$]
{Torus quotients of Richardson varieties in $G_{r,qr+1}$}
\author{S. Senthamarai Kannan}
\address{%
	S. Senthamarai Kannan\\
	Chennai Mathematical Institute\\
	Plot H1, SIPCOT IT Park\\
	Siruseri, Kelambakkam\\
	603103, India\\
	Email:kannan@cmi.ac.in
}
\author{Arpita Nayek}
\address{%
	Arpita Nayek\\
	IIT Bombay\\
	Powai, Mumbai\\
	400076, India\\
	Email:arpitan@math.iitb.ac.in
}

\begin{document}
	\maketitle
	
	\begin{abstract}
		Let $r$ and $q$ be positive integers and $n=qr+1.$ Let $G = SL(n, \mathbb{C})$ and $T$ be a maximal torus of $G.$ Let $P^{\alpha_r}$ be the maximal parabolic subgroup of $G$ corresponding to the simple root $\alpha_r.$ Let $\omega_r$ be the fundamental weight corresponding to $\alpha_r.$ Let $W$ be the Weyl group of $G$ and $W_{P^{\alpha_r}}$ be the Weyl group of $P^{\alpha_r}.$ Let $W^{P^{\alpha_r}}$ be the set of all minimal coset representatives of $W/W_{P^{\alpha_r}}$ in $W.$ Let $w_{r,n}$ (respectively, $v_{r,n}$) be the minimal (respectively, maximal) element in $W^{P^{\alpha_{r}}}$ such that $w_{r,n}(n\omega_r) \leq 0$ (respectively, $v_{r,n}(n\omega_r) \geq 0$). Let $v \leq v_{r,n}$ and $X^v_{w_{r,n}}$ be the Richardson variety in $G_{r,n}$ corresponding to $v$ and $w_{r,n}.$  In this article, we give a sufficient condition on $v$ such that the GIT quotient of $X^{v}_{w_{r,n}}$ for the action of $T$ is the product of projective spaces with respect to the descent of the line bundle $\mathcal{L}(n\omega_r).$ 
	\end{abstract}
	
	
	
	\section{Introduction}
	
	For the action of  $T$ on the Grassmannian $G_{r,n},$ the GIT quotients have been studied extensively. In \cite{HK}, Hausmann and Knutson identified the GIT quotient of the Grassmannian  $G_{2,n}$ by $T$  with the moduli space of polygons in $\mathbb{R}^{3}.$ Also, they showed that GIT quotient of $G_{2,n}$ by $T$ can be realized as the GIT quotient of an $n$-fold product of projective lines by the diagonal action of $PSL(2, \mathbb{C}).$ In the symplectic geometry literature, these spaces are known as polygon spaces as they parameterize the $n$-sides polygons in $\mathbb{R}^{3}$ with fixed edge length up to rotation. More generally, GIT quotient of $G_{r,n}$ by $T$ can be identified with the GIT quotient of $(\mathbb{P}^{r-1})^{n}$ by the diagonal action of $PSL(r, \mathbb{C})$ called the Gel'fand-Macpherson correspondence.
	
	For a parabolic subgroup $Q$ of $G,$ Howard \cite{How} considered the problem of determining which line bundles on $G/Q$ descend to ample line bundles of the GIT quotient of $G/Q$ by $T.$ Howard showed that when $\mathcal{L}(\lambda)$ is a very ample line bundle on $G/Q$ (so the character $\lambda$ of $T$ extends to $Q$ and to no larger subgroup of $G$) and $H^{0}(G/Q,\mathcal{L}(\lambda))^{T}$ is non-zero, the line bundle descends to the quotient \cite{How}. He extended the results to the case when the $T$-linearization of $\mathcal{L}(\lambda)$ is twisted by $\mu,$ a character of $T.$ He proved that the line bundle $\mathcal{L}(\lambda)$ twisted by $\mu$ descends to the GIT quotient provided the $\mu$-weight space of $H^{0}(G/Q, \mathcal{L}(\lambda))$ is non-zero and this is so when $\lambda-\mu$ is in the root lattice and $\mu$ is in the convex hull of the Weyl group orbit of $\lambda.$ This was extended to any simple algebraic group by Kumar \cite{Kum}.   
	
	In \cite{Sko}, Skorobogatov gave a combinatorial description using Hilbert-Mumford criterion, when a point in $G_{r,n}$ is semistable with respect to the $T$-linearized line bundle $\mathcal{L}(\omega_r).$  As a corollary he showed that when $r$ and $n$ are coprime semistability is the same as stability. Independently, in \cite{K1} and \cite{K2}, first named author gave a description of parabolic subgroups $Q$ of a simple algebraic group $G$ for which there exists an ample line bundle $\mathcal{L}$ on $G/Q$ such that $(G/Q)^{ss}_T(\mathcal{L})$ is the same as $(G/Q)^{s}_T(\mathcal{L}).$ In particular, in the case when $G=SL(n,\mathbb{C})$ and $Q=P^{\alpha_{r}},$ Kannan showed that $(G_{r,n})^{ss}_T(\mathcal{L})$ is the same as $(G_{r,n})^{s}_T(\mathcal{L})$ if and only if $r$ and $n$ are coprime.

	In \cite{KS}, first named author and Sardar showed that there exists a unique minimal dimensional Schubert variety $X(w_{r,n})$ in $G_{r,n}$ admitting semistable points with respect to the $T$-lineararized line bundle $\mathcal{L}(\omega_{r}).$ They showed that $w_{r,n}$ is determined by the combinatorial property: minimal length element $w$ in $W^{P^{\alpha_r}}$ such that $w(n\omega_r) \leq 0.$
	
	In order to state the main result of our article we introduce a few notation. 
	
	Note that by using the above combinatorial condition $w_{r,n}=\prod_{i=1}^{r}(s_{iq}s_{iq-1}\cdots s_{i+1}s_{i}).$ In one line notation $w_{r,n}$ is $(q+1,2q+1,3q+1,\ldots,rq+1).$ Let $v_{r,n} \in W^{P^{\alpha_r}}$ be the unique maximal element such that $v_{r,n}(n\omega_r)\geq 0$. Then we see that $v_{r,n}=\prod_{i=2}^{r}(s_{(i-1)q}s_{(i-1)q-1}\cdots s_{i+1}s_{i})$ and in one line notation $v_{r,n}$ is $(1,q+1,2q+1,\ldots,(r-1)q+1).$ Note that $w_{r,n}=cv_{r,n}$ for a Coxeter element $c=\prod_{i=1}^{r}(s_{iq}s_{iq-1}\cdots s_{(i-1)q+2}s_{(i-1)q+1})$ such that $l(w_{r,n})=l(v_{r,n})+(n-1)$. 
	
	For a sequence of integers $1 \leq m_1, m_2, \ldots, m_{r-1} \leq q,$ let $\underline{m}=(m_1, m_2, \ldots, m_{r-1})$ and let
	$v_{\underline{m}} = \prod_{i=1}^{r-1}(s_{(i-1)q+m_{i}}s_{(i-1)q+m_{i}-1}\cdots s_{i+2}s_{i+1}).$ In one line notation $v_{\underline{m}}$ is $(1, m_{1}+1, q+m_{2}+1, \ldots, (r-2)q+m_{r-1}+1).$ Note that $v_{\underline{m}} \leq v_{r,n}.$ By \cite[Proposition 3.1, p.256]{kannanetal}, $ (X^{v_{\underline{m}}}_{w_{r,n}})^{ss}_T(\mathcal{L}(n\omega_r))$ is non-empty.
	
	In this article, we  study the torus quotients of the Richardson varieties $X^{v_{\underline{m}}}_{w_{r,n}}$ with respect to the descent of the ample line bundle $\mathcal{L}(n\omega_r).$ The main theorem of this article is the following:
	\begin{theorem}$($see \cref{theorem8.1}$)$\label{theorem1.1}
		The GIT quotient $T\backslash \backslash (X^{v_{\underline{m}}}_{w_{r,n}})^{ss}_T(\mathcal{L}(n\omega_r))$ is isomorphic to ${\mathbb P}^{q-m_1} \times {\mathbb P}^{q-m_2} \times \cdots \times {\mathbb P}^{q-m_{r-1}}$ and the GIT quotient is embedded via the very ample line bundle $\mathcal{O}_{\mathbb{P}^{q-m_1}}(r-1) \boxtimes \mathcal{O}_{\mathbb{P}^{q-m_2}}(r-2) \boxtimes \cdots \boxtimes \mathcal{O}_{\mathbb{P}^{q-m_{r-1}}}(1)$. 
	\end{theorem}
	
	The layout of the article is as follows. In \cref{section2}, we recall some preliminaries on algebraic groups, Lie algebras, Standard Monomial Theory, Geometric Invariant Theory and Deodhar components. In \cref{section3}, we make some observations about the standard tableau associated to a standard monomial in $H^0(X^{v_{\underline{m}}}_{w_{r,n}}, \mathcal{L}(n\omega_r)^{\otimes k})^{T}.$ Also, we prove that the homogeneous coordinate ring of $T\backslash \backslash (X^{v_{\underline{m}}}_{w_{r,n}})^{ss}_T(\mathcal{L}(n\omega_r))$ is generated by $H^0(X^{v_{\underline{m}}}_{w_{r,n}}, \mathcal{L}(n\omega_r))^{T}$ as a $\mathbb{C}$-algebra. In \cref{section5}, we study the Deodhar component $\mathcal{R}^{v_{\underline{m}}}_{w_{r,n}}$ in the Richardson variety $X^{v_{\underline{m}}}_{w_{r,n}}.$ Further, we find the common factor of the restriction of all $T$-invariant standard monomials to $\mathcal{R}^{v_{\underline{m}}}_{w_{r,n}}$ and the non-common factor of the restriction of a $T$-invariant standard monomial to $\mathcal{R}^{v_{\underline{m}}}_{w_{r,n}}.$ In \cref{section6}, we prove \cref{theorem1.1} (see \cref{theorem8.1}).

	\section{Notation and preliminaries}\label{section2}
	
	In this section, we set up some notation and preliminaries. We refer to \cite{R}, \cite{Hum1}, \cite{Hum2},  \cite{LS}, \cite{Mumford}, \cite{New} for preliminaries.
	
	Let $G=SL(n,\mathbb{C})$ and $T$ be a maximal torus in $G.$ Let $R$ be the set of roots with respect to $T.$ Let $B$ (respectively, $B^{-}$) be the upper (respectively, lower) triangular matrices in $G$. Let $U$ and $U^{-}$ be the unipotent radical of $B$ and $B^{-}$ respectively. Let $R^{+}$ (respectively, $R^{-}$) be the positive roots with respect to $(T, B).$ Let $S= \{\alpha_1,\ldots,\alpha_{n-1}\}$ denote the set of simple roots in $R^{+}$. Let $W=N_{G}(T)/T$ denote the Weyl group of $G$ with respect to $T.$ Let $s_{i}$ be the simple reflection corresponding to $\alpha_i.$  For $\alpha_r \in S$, the subgroup of $G$ generated by $B$ and $\{n_{\alpha_i} : i \neq r\}$ is called the maximal parabolic subgroup of $G$ corresponding to $\alpha_r$, where $n_{\alpha_i}$ is a representative of $s_i$ in $N_G(T).$ We denote it by $P^{\alpha_r}.$ The Weyl group $W_{P^{\alpha_r}}$ of $P^{\alpha_r}$ is the subgroup of $W$ generated by $\{s_i : i \neq r\}.$ Let $W^{P^{\alpha_r}} = \{w \in W : w(\alpha) \in R^{+} \text{ for all } \alpha \in S \setminus \{\alpha_r\}\}$ be the set of all minimal coset representatives of $W/W_{P^{\alpha_r}}$ in $W.$ Then  $G/P^{\alpha_{r}}$ is the Grassmannian $G_{r,n}$ of all $r$-dimensional subspaces of $\mathbb{C}^{n}.$ 
	
	Let $I(r,n)=\{(a_{1},a_{2},\ldots, a_{r}) \in \mathbb{N}^r: 1\le a_{1}<a_{2}<\cdots <a_{r}\le n\}.$ There is a natural order on $I(r,n),$ given by $(a_{1},a_{2},\ldots, a_{r})\le (b_{1},b_{2},\ldots, b_{r})$ if and only if $a_{i} \leq b_{i}$ for all $1\le i\le r.$ Consider the Bruhat order on $W^{P^{\alpha_r}}.$ Then there is an order preserving identification of $W^{P^{\alpha_{r}}}$ with $I(r,n),$ the correspondence is given by $w\in W^{P^{\alpha_{r}}}$ mapping to $(w(1),w(2),\ldots, w(r)).$ 
	
	Let $V=\mathbb{C}^{n}.$ Let $\{e_{1},e_{2},\ldots,e_{n}\}$ be the standard basis of $V.$ Let $V_{r}$ be the subspace of $V$ spanned by $\{e_{1},e_{2},\ldots, e_{r}\}.$ Then there is a natural projective variety structure on $G_{r,n}$ given by the Pl\"ucker embedding $\phi: G_{r,n}\longrightarrow \mathbb{P}(\wedge^{r}V)$ sending $r$-dimensional subspace to its $r$-th exterior power. Furthermore, there is a $G$-equivariant isomorphism of $G/P^{\alpha_{r}}$ with $G_{r,n},$ sending $gP^{\alpha_{r}}$ to $gV_{r}.$ For the natural action of $T$ on $G/P^{\alpha_{r}},$ the $T$-fixed points of $G/P^{\alpha_{r}}$ are identified with the elements of $W^{P^{\alpha_{r}}},$ sending $w$ to $wP^{\alpha_{r}}.$ For $w\in W^{P^{\alpha_r}},$ the closure of the $B$-orbit passing through $wP^{\alpha_{r}}$ in $G/P^{\alpha_{r}}$ has a natural structure of a projective variety called Schubert variety associated to $w,$ denoted by $X(w).$ Let $e_{\underline{i}}=e_{i_1} \wedge e_{i_2} \wedge \cdots \wedge e_{i_r},$ for $\underline{i}=(i_1, i_2, \ldots, i_r) \in I(r,n).$ Then $\{e_{\underline{i}}: \underline{i} \in I(r,n)\}$ forms a basis of $\wedge^rV.$ Let $\{p_{\underline{i}}: \underline{i}\in I(r,n)\} \subseteq (\wedge^rV)^{*}$ be the dual basis of the basis $\{e_{\underline{i}}: \underline{i} \in I(r,n)\}.$ Then $p_{\underline{i}}$'s are called the Pl\"{u}cker coordinates. Note that $p_{\underline{i}}|_{X(w)} \neq 0$ if and only if $\underline{i} \leq (w(1), w(2), \ldots, w(r)).$ The restriction of a monomial $p_{\tau_1}p_{\tau_2}\ldots p_{\tau_m} \in H^0(G_{r,n},\mathcal{L}(\omega_r)^{\otimes m})$ to $X(w),$ where $\tau_1, \tau_2, \ldots, \tau_m \in I(r,n)$ is said to be a standard monomial of degree $m$ if $\tau_1 \leq \tau_2 \leq \cdots \leq \tau_m \leq w.$ The standard monomials of degree $m$ on $X(w)$ form a basis of $H^0(X(w),\mathcal{L}(\omega_r)^{\otimes m})$ (see \cite[Theorem 4.5.0.5, p.43]{LS}).
	
	Let $v, w \in W$ and let $X^v:=\overline{B^{-}vP^{\alpha_r}/P^{\alpha_r}}$ be the opposite Schubert variety in $G/P^{\alpha_r}$ corresponding to $v.$ Then a Richardson variety $X^v_w$ in $G/P^{\alpha_r}$ is the intersection of $X(w)$ and $X^v$ in $G/P^{\alpha_r}$ with a reduced variety structure. Richardson varieties are $T$-stable irreducible varieties. Note that $X^v_w$ is non-empty if and only if $v \leq w$ (see \cite[Lemma 1, p.655]{BL}).
	
	Let $X(T)$ (respectively, $Y(T)$) denote the group of all characters (respectively, one-parameter subgroups) of $T.$  Let $E_1:=X(T) \otimes \mathbb{R},$ $E_2:=Y(T) \otimes \mathbb{R}.$ Let $\langle \cdot , \cdot \rangle: E_1 \times E_2 \longrightarrow \mathbb{R}$ be the canonical non-degenerate bilinear form. Let $\bar{C}:=\{\lambda \in E_2 : \langle \alpha, \lambda \rangle \geq 0, \text{ for all } \alpha \in R^{+}\}.$ Note that for each $\alpha \in R,$ there is a homomorphism $\phi_{\alpha}:SL(2,\mathbb{C}) \longrightarrow G$ such that $\check{\alpha}:\mathbb{G}_m \longrightarrow G$ is given by $\check{\alpha}(t)=\phi_{\alpha}(\begin{pmatrix}
		t & 0\\
		0 & t^{-1}
	\end{pmatrix})$ (see \cite[Section 1.9, p.19]{carter}). Consider the homomorphisms $x_i: \mathbb{G}_{a} \to U$ and $y_i: \mathbb{G}_a \to U^{-}$ defined by
	$	x_{i}(m)=\phi_{\alpha_i}\begin{pmatrix}
		1 & m\\
		0 & 1
	\end{pmatrix} \text{and } y_{i}(m)=\phi_{\alpha_i}\begin{pmatrix}
		1 & 0\\
		m & 1
	\end{pmatrix}$
	for all $m \in \mathbb{C}.$ We also have $s_{\alpha}(\chi)=\chi - \langle \chi, \check{\alpha} \rangle \alpha$ for all $\alpha \in R$ and $\chi \in E_1.$ Let $\{\omega_i: i= 1,2, \ldots, n-1\} \subset E_1$ be the fundamental weights; i.e. $\langle \omega_i, \check{\alpha_j} \rangle=\delta_{ij}$ for all $i,j=1, 2, \ldots, n-1.$ There is a natural partial order $\leq$ on $X(T)$ defined by $\mu\le \lambda $ if and only if $\lambda -\mu$ is a non-negative integral linear combination of simple roots.
	
	Let $\mathcal{L}$ be a ample line bundle on $G/P^{\alpha_r}.$ Note that $\mathcal{L}$ is $T$-linearized (see \cite[Remark, p.67]{knop}). We also denote the restriction of the line bundle $\mathcal{L}$ on $X(w)$ by $\mathcal{L}.$ A point $x \in X(w)$ is said to be semistable with respect to the $T$-linearized line bundle $\mathcal{L}$ if there is a $T$-invariant section $s \in H^0(X(w),\mathcal{L}^{\otimes m})$ for some positive integer $m$ such that $s(x)\neq 0.$ We denote the set of all semistable points of $X(w)$ with respect to $\mathcal{L}$ by $X(w)^{ss}_T(\mathcal{L}).$ A point $x$ in $X(w)^{ss}_{T}(\mathcal{L})$ is said to be a stable point if the $T$-orbit of $x$ is closed in $X(w)^{ss}_{T}(\mathcal{L})$ and stabilizer of $x$ in $T$ is finite. We denote the set of all stable points of $X(w)$ with respect to $\mathcal{L}$ by $X(w)^{s}_T(\mathcal{L})$ (see \cite{Mumford}, \cite{New}). 
	
	Let $\lambda$ be a dominant character of $T$ with respect to $G.$ Let $V(\lambda)$ be the irreducible representation of $G$ with highest weight $\lambda$.  Let $P_{\lambda}$ be the stabiliser in $G$ of the highest weight line in $V(\lambda).$ Let $\mathcal{L}_{\lambda}$ be the homogeneous ample line bundle on $G/P_{\lambda}$ associated to $\lambda$. The $T$-linearization of $\mathcal{L}_{\lambda}$ is given by restricting the action of $G$ on $\mathcal{L}_{\lambda}$ to $T.$ The following theorem describes which line bundles descend to the GIT  quotient $T \backslash\backslash (G/P_{\lambda})^{ss}_T(\mathcal{L}_\lambda)$ (see \cite[Theorem 3.10, p.764]{Kum}). Kumar's result in \cite{Kum} is more general than what is presented here.
	
	\begin{theorem}$($See \cite[Theorem 3.10, p.764]{Kum}$)$\label{shrawan} Let $G=SL(n, \mathbb{C}).$ Then the line bundle $\mathcal{L}_{\lambda}$ descends to a line bundle on the GIT quotient $T \backslash\backslash (G/P_{\lambda})^{ss}_T(\mathcal{L}_\lambda)$ if and only if $\lambda$ lies in the root lattice. 
	\end{theorem} 
	
	We recall some basic facts about standard Young tableau for generalized flag varieties $G/P^{I},$ where $I$ is a subset of $S$ (see \cite[p.216]{LB}). Let $\lambda=\Sigma_{i=1}^{n-1} a_i\omega_i$, $a_i \in \mathbb{Z}_{\geq 0}$ be a dominant weight. To $\lambda$ we associate a Young diagram (denoted by $\Gamma$) with $\lambda_i$ number of boxes in the $i$-th column, where $\lambda_i:=a_i+\ldots+a_{n-1}$, $1 \leq i \leq n-1$.
	
	A Young diagram $\Gamma$ associated to a dominant weight $\lambda$ is said to be a Young tableau if the diagram is filled with integers $1, 2, \ldots, n$. We also denote this Young tableau by $\Gamma$. A Young tableau is said to be standard if the entries along any column is non-decreasing from top to bottom and along any row is strictly increasing from left to right. 
	
	Given a Young tableau $\Gamma$, let $\tau=\{i_1,i_2,\ldots,i_d\}$ be a typical row in $\Gamma$, where $1 \leq i_1 < \cdots < i_d \leq n$, for some $1 \leq d \leq n-1$. To the row $\tau$, we associate the Pl\"{u}cker coordinate $p_{i_1,i_2, \ldots,i_d}$. We set $p_{\Gamma}=\prod_{\tau}p_{\tau}$, where the product is taken over all the rows of $\Gamma$. We say that $p_{\Gamma}$ is a standard monomial on $G/P_{\lambda}$ if $\Gamma$ is standard.
	
	Now we recall the definition of weight of a standard Young tableau $\Gamma$ (see \cite[Section 2, p.336]{littelmann}). For a positive integer $1\leq i\leq n$, we denote by $c_{\Gamma}(i)$,
	the number of boxes of $\Gamma$ containing the integer $i$. Let $\epsilon_i: T\to \mathbb G_m$  
	be the character defined as $\epsilon_i(diag(t_1,\ldots, t_n))=t_i$. 
	We define the weight of $\Gamma$ as 
	$$wt(\Gamma):=c_{\Gamma}(1)\epsilon_1+ \cdots + c_{\Gamma}(n)\epsilon_n.$$
	
	Now we recall the following lemma about $T$-invariant monomials in $H^0(G/P_{\lambda}, \mathcal L_{\lambda}).$
	
	\begin{lemma}$($See \cite[Lemma 3.1, p.4]{NP}$)$\label{lemma2.2}
		A monomial $p_{\Gamma}\in H^0(G/P_{\lambda}, \mathcal L_{\lambda})$ is $T$-invariant if and only if  all the entries in $\Gamma$ appear equal number of times. 
	\end{lemma}

	\subsection{Some preliminaries on Deodhar component}
	In this subsection we recall some preliminaries on Deodhar component. In \cite{deodhar} Deodhar considered the intersection in $G/B$ of the open cell in a Schubert variety with the open cell of an opposite Schubert variety. For $v, w \in W,$ define the Richardson strata $$R^v_w=(BwB/B) \cap (B^{-}vB/B).$$
	
	Recall that for $v, w \in W$ a Richardson variety in $G/B$ is defined to be the intersection of $X(w) \cap X^v.$ Since both $X(w)$ and $X^v$ contain the intersection of $(BwB/B) \cap (B^{-}vB/B).$ it is clear that $R_w^v \subseteq X_w^v.$ And so Richardson strata is empty if $v \nleq w$ and the closure of $R_w^v$ is $X_w^v.$
	
	In \cite{deodhar} Deodhar gave a refined decomposition of a Richardson strata in $G/B$ into disjoint locally closed subvarieties of a Schubert variety. We follow the notation from Marsh and Reitsch \cite{marsh}, and Kodama and Williams \cite{kodama}. The definitions and examples are taken verbatim from \cite{kodama} since it is their notation and set up that we use in our proofs.
	
	Fix a reduced expression $\mathbf{w}= s_{i_1}s_{i_2}\cdots s_{i_m}.$ We define a subexpression $\mathbf{v}$ of $\mathbf{w}$ to be a word obtained from the reduced expression $\mathbf{w}$ by replacing some of the factors with $1.$ For example, consider a reduced expression in $S_4,$ say $s_3s_2s_1 s_3s_2s_3.$ Then $s_3s_21s_3s_21$ is a subexpression of $s_3s_2s_1s_3s_2s_3.$ Given a subexpression $\mathbf{v},$ we set $\mathbf{v}_{(k)}$ to be the product of the leftmost $k$ factors of $\mathbf{v},$ if $k \geq 1,$ and set $\mathbf{v}_{(0)}= 1.$ The following definition was given in \cite{marsh} and was inspired from Deodhar’s paper \cite{deodhar}.
	
	\begin{definition}
		Given a subexpression $\mathbf{v}$ of a reduced expression $\mathbf{w} = s_{i_1}s_{i_2}\cdots s_{i_m},$ we define
		\begin{center}
			$J_\mathbf{v}^{\circ}:= \{k \in \{1, \ldots, m\}| \mathbf{v}_{(k-1)} < \mathbf{v}_{(k)}\}$\\
			$J_\mathbf{v}^{\Box}:= \{k \in \{1, \ldots, m\}| \mathbf{v}_{(k-1)} = \mathbf{v}_{(k)}\}$\\
			$J_\mathbf{v}^{\bullet}:= \{k \in \{1, \ldots, m\}| \mathbf{v}_{(k-1)} > \mathbf{v}_{(k)}\}.$
		\end{center}
		The expression $\mathbf{v}$ is called non-decreasing if $\mathbf{v}_{(j-1)} \leq \mathbf{v}_{(j)}$ for all $j = 1, \ldots, m,$ and in this case $J_\mathbf{v}^{\bullet}=\emptyset.$
	\end{definition}
	The following definition is from \cite[Definition 2.3]{deodhar}.
	
	\begin{definition}[Distinguished subexpressions] A subexpression $\mathbf{v}$ of $\mathbf{w}$ is called distinguished if we have
		$$\mathbf{v}_{(j)} \leq \mathbf{v}_{(j-1)} s_{i_j} ~\forall j \in \{1, \ldots, m\}.$$
	\end{definition}
	
	In other words, if right multiplication by $s_{i_j}$ decreases the length of $\mathbf{v}_{(j-1)},$ then in a distinguished subexpression we must have $\mathbf{v}_{(j)} = \mathbf{v}_{(j-1)}s_{i_j}.$
	
	We write $\mathbf{v} \prec \mathbf{w}$ if $\mathbf{v} $ is a distinguished subexpression of $\mathbf{w}.$
	
	\begin{definition}[Positive distinguished subexpressions] We call a subexpression $\mathbf{v} $ of $\mathbf{w}$ a positive distinguished subexpression (or a PDS for short) if $\mathbf{v}_{(j-1)} < \mathbf{v}_{(j-1)}s_{i_j},$ for all $j \in \{1, \ldots, m\}.$
	\end{definition}
	
	In \cite{marsh} Marsh and Rietsch proved the following.
	
	\begin{lemma} Given $v \leq w$ and a reduced expression $\mathbf{w} = s_{i_1}s_{i_2}\cdots s_{i_m}$ for $w,$ there is a unique PDS $\mathbf{v} ^{+}$ for $v$ in $\mathbf{w}.$
	\end{lemma}
	We now describe the Deodhar decomposition of the Richardson strata. Marsh and Rietsch \cite{marsh} gave explicit parameterizations for each Deodhar component, identifying each one with a subset in the group. Much of this appears implicitly in Deodhar’s paper, but we refer to \cite{marsh} for our exposition because these statements are made explicit there and the authors make references to Deodhar’s paper wherever needed.
	
	Let $x_i$ and $y_i$ be the homomorphisms as in paragraph no. $4$ of page no. $3$.
	
	\begin{definition}\cite[Definition 5.1]{marsh} Let $\mathbf{w} = s_{i_1}s_{i_2}\cdots s_{i_m}$ be a reduced expression for $w,$ and
		let $\mathbf{v}$ be a distinguished subexpression. Define a subset $\mathbf{G^v_w}$ in $G$ by
		\begin{equation}\mathbf{G^v_w}:=\{g=g_1g_2\cdots g_m|\Bigg\{
			\left.\begin{array}{lr}
				g_l=x_{i_l}(m_l)s_{i_l} & \text{if $l \in J_\mathbf{v}^{\bullet}$}\\
				g_l=y_{i_l}(p_l) & \text{if $l \in J_\mathbf{v}^{\Box}$}\\
				g_l=s_{i_l} & \text{if $l \in J_\mathbf{v}^{\circ}$}\\
			\end{array}\right. \text{ where } l \in \{1,2,\ldots,m\} \text{ and } m_l, p_l \in \mathbb{C}\}.
		\end{equation}
		
	\end{definition}
	From \cite[Theorem 4.2]{marsh} there is an isomorphism from $\mathbb{C}^{*|J_\mathbf{v}^{\Box}|} \times \mathbb{C}^{|J_\mathbf{v}^{\circ}|}$ to $\mathbf{G^v_w}.$
	
	\begin{definition}[Deodhar Component] The Deodhar component $\mathcal{R}^{\mathbf{v}}_{\mathbf{w}}$ is the image of $\mathbf{G^v_w}$
		under the map $\mathbf{G^v_w} \subseteq U^{-}vB \cap BwB \longrightarrow G/B$ sending $g$ to $gB.$
	\end{definition}
	
	Then from \cite[Theorem 1.1]{deodhar} one has \cite[Corollary 1.2]{deodhar}, also from Deodhar.

	\begin{definition}[Deodhar Component in $G/P^{\alpha_r}$] Let $\mathbf{v}, \mathbf{w} \in W^{P^{\alpha_r}}.$ Then the Deodhar component $\mathcal{R}^{\mathbf{v}}_{\mathbf{w}}$ is the image of $\mathbf{G^v_w}$
		under the map $\mathbf{G^v_w} \subseteq U^{-}vP^{\alpha_r} \cap BwP^{\alpha_r} \longrightarrow G/P^{\alpha_r}$ sending $g$ to $gP^{\alpha_r}.$
	\end{definition}
	
	\begin{theorem} $R^v_w = \bigsqcup_{\mathbf{v} \prec \mathbf{w}}\mathcal{R}^{\mathbf{v}}_{\mathbf{w}}$ the union taken over all distinguished subexpressions $\mathbf{v}$ such
		that $\mathbf{v}_{(m)}=v.$ The component $\mathcal{R}^{\mathbf{v}^{+}}_\mathbf{w}$ is open in $R^v_w.$
	\end{theorem}
	
	Naturally when one is talking of the Deodhar decomposition of a Richardson strata in $G/P^{\alpha_r}$ one can take the projections of the components in $G/B$ into $G/P^{\alpha_r}.$ In \cite[Proposition 4.16]{kodama} the authors show that the Deodhar components of a Richardson strata in $G/P^{\alpha_r}$ are independent of $\mathbf{w}$ and only depends upon $w.$ This follows from the observation
	that any two reduced decompositions $\mathbf{w}$ and $\mathbf{w}'$ of $w$ are related by a sequence of commuting transpositions $s_is_j = s_js_i.$
	
	\section{Some combinatorics of Young tableau and $R_1$ generation}\label{section3}
	In this section we make some observations about the standard tableau associated to a standard monomial in $H^0(X^{v_{\underline{m}}}_{w_{r,n}}, \mathcal{L}^{\otimes k})^{T},$ where $\mathcal{L}:=\mathcal{L}(n\omega_r).$ Further, we prove that the homogeneous coordinate ring of $T\backslash \backslash (X^{v_{\underline{m}}}_{w_{r,n}})^{ss}_T(\mathcal{L})$ is generated by $H^0(X^{v_{\underline{m}}}_{w_{r,n}}, \mathcal{L})^{T}$ as a $\mathbb{C}$-algebra.  Let $R_k=H^0(X^{v_{\underline{m}}}_{w_{r,n}}, \mathcal{L}^{\otimes k})^{T}$ $(k \in \mathbb{Z}_{\geq 0}).$ 
	
	\subsection{Some combinatorics of Young tableau}\label{subsection3.1}Let $f \in R_k$ be a standard monomial and $\Gamma$ be the standard Young tableau associated to $f.$ Then $\Gamma$ has $nk$ rows and $r$ columns. Since $f$ is $T$-invariant, by \cref{lemma2.2} we have $c_{\Gamma}(t)=rk$ for all $1 \leq t \leq n,$  
	where $c_{\Gamma}(t)$ is the number of boxes in $\Gamma$ containing $t.$ Let $row_i$ be the $i$-th row of $\Gamma$. Let $E_{i,j}$  be the $(i,j)$-th entry of the tableau $\Gamma$ and $N_{t,j}$ be the number of boxes in the $j$-th column of $\Gamma$ containing $t.$ Since we are working with $X^{v_{\underline{m}}}_{w_{r,n}},$ we have $row_1 \geq v_{\underline{m}}$ and $row_n \leq w_{r,n}.$ Thus, we have
	\begin{align}
		E_{1,j} \geq (j-2)q+m_{j-1}+1 \text{ for all } 2 \leq j \leq r \text{ and } E_{n,j} \leq jq+1 \text{ for all } 1 \leq j \leq r.
		\label{(6.1*)}
	\end{align}
	\begin{lemma}\label{lemma3.1}
		Let $j$ be a fixed integer such that $1 \leq j \leq r-1.$ Then every integer between $(j-1)q+2$ and $jq+1$ appears either in the $j$-th column or in the $j+1$-th column of $\Gamma.$ Also, every integer between $(r-1)q+2$ and $rq+1$ appears in the $r$-th column of $\Gamma.$
	\end{lemma}
	\begin{proof}Assume that $j=1.$ Since $E_{1,3} \geq q+m_{2}+1 \geq q+2$ (see \eqref{(6.1*)}), every integer between $1$ and $q+1$ appears either in the first column or in the second column. Consider $2 \leq j \leq r-2.$ Since $E_{1,j+2} \geq jq+m_{j+1}+1 \geq jq+2$ and $E_{n,j-1} \leq (j-1)q+1$ (see \eqref{(6.1*)}), every integer between $(j-1)q+2$ and $jq+1$ appears either in the $j$-th column or in the $j+1$-th column. Now we assume that $j=r-1.$ Since $E_{n,r-2} \leq (r-2)q+1,$ every integer between $(r-2)q+2$ and $(r-1)q+1$ appears either in the $r-1$-th column or in the $r$-th column. 
		
		Since $E_{n,r-1} \leq (r-1)q+1,$ every integer between $(r-1)q+2$ and $rq+1$ appears in the $r$-th column.
	\end{proof}
	\begin{lemma} \label{lemma3.2}
		Let $j$ be a fixed integer such that $1 \leq j \leq r-1.$ Then we have \begin{itemize}
			\item[(i)] $(j-1)q+m_{j}+1 \leq E_{i,j+1} \leq jq+1$ for all $1 \leq i \leq k(r-j).$
			\item[(ii)] $E_{k(r-j)+1,j+1} \geq jq+2.$
		\end{itemize}
	\end{lemma}
	\begin{proof}
		Assume that $1 \leq j \leq r-2.$ Since $E_{1,j+2}\geq jq+m_{j+1}+1 \geq jq+2,$ every integer between $1$ and $jq+1$ appears in the first $j+1$ columns of $\Gamma$. For $j=r-1,$ every integer between $1$ and $(r-1)q+1$ appears in $r$ columns of $\Gamma$. Further note that for $1 \leq j \leq r-1,$ $\sum_{l=1}^{j}\sum_{t=1}^{jq+1}N_{t,l}=jk(rq+1)$ and $\sum_{t=1}^{jq+1}c_{\Gamma}(t)=rk(jq+1).$ Thus, $\sum_{t=1}^{jq+1}N_{t,j+1} = rk(jq+1)-jk(rq+1)=k(r-j).$ Hence, $E_{k(r-j),j+1} \leq jq+1$ and $E_{k(r-j)+1,j+1} \geq jq+2.$ In particular, since $E_{1,j+1} \geq (j-1)q+m_j+1,$ we have $(j-1)q+m_{j}+1 \leq E_{i,j+1} \leq jq+1$ for all $1 \leq i \leq k(r-j).$
	\end{proof}
	\begin{corollary}\label{corollary3.3} We have $E_{k,r} \leq (r-1)q+1$ and $E_{m,r}=(r-1)q+l,$ where $l \geq 2$ is the least positive integer such that $k(r(l-1)+1) \geq m.$
	\end{corollary}
	\begin{proof}
		By \cref{lemma3.2} we have $E_{k,r} \leq (r-1)q+1$ and $E_{k+1,r} \geq (r-1)q+2.$ By \cref{lemma3.1} we have every integer between $(r-1)q+2$ and $rq+1$ appears in the $r$-th column of $\Gamma.$ Thus, $E_{m,r}=(r-1)q+l,$ where $l \geq 2$ is the least positive integer such that $k(r(l-1)+1) \geq m.$  
	\end{proof}
	\begin{lemma}\label{lemma3.4}
		Let $j$ be a fixed integer such that $1 \leq j \leq r-1.$ Then we have \begin{itemize}
			\item[(i)] For $2 \leq l \leq m_j,$ 
			\begin{align*}
				E_{m,j}=(j-1)q+l \text{ for all } k(r(l-1)-j+1)+1 \leq m \leq k(rl-j+1).
			\end{align*} 
			\item[(ii)] For $m_j+1 \leq l \leq q+1,$ \begin{align*}
				E_{m,j}=(j-1)q+l \text{ for all } k\big(r(l-1)-j+1\big)+1 \leq m \leq k\big(r(l-1)+1\big).
			\end{align*} 
		\end{itemize}
	\end{lemma}
	\begin{proof}
		Proof of (i):
		Since $E_{1,j+1}\geq (j-1)q+m_j+1,$ by \cref{lemma3.1} every integer between $(j-1)q+2$ and $(j-1)q+m_j$ appears $kr$ times in the $j$-th column. Since by \cref{lemma3.2}
		$E_{k(r-j+1),j} \leq (j-1)q+1,$ for $2 \leq l \leq m_j$ we have 
		$E_{m,j}=(j-1)q+l$ for all $k(r(l-1)-j+1)+1 \leq m \leq k(rl-j+1).$ 
		
		Proof of (ii): Assume that $m_j+1 \leq l \leq q+1.$ 
		Since by (i) $E_{k(m_jr-j+1),j}\leq(j-1)q+m_j$  and $\sum_{t=m_j+1}^{l-1}N_{(j-1)q+t,j} \leq kr(l-m_j-1),$ we have $E_{m,j} \geq (j-1)q+l$ for all $m \geq k(m_jr-j+1)+kr(l-m_j-1)+1=k(r(l-1)-j+1)+1.$ Now it is enough to prove that $E_{m,j} \leq (j-1)q+l$ for all $m \leq k(r(l-1)+1).$ Since $\sum_{t=m_j+1}^l N_{(j-1)q+t,j+1}\leq k(r-j),$ we have $\sum_{t=m_j+1}^l N_{(j-1)q+t,j} \geq kr(l-m_j)-k(r-j).$ Also since $E_{k(m_jr-j+1),j} \leq (j-1)q+m_j,$ we have $E_{m,j} \leq (j-1)q+l$ for all $m \leq k(m_jr-j+1)+kr(l-m_j)-k(r-j)=k(r(l-1)+1).$
	\end{proof}	
	\begin{corollary}\label{corollary3.5}
		Let $j$ be a fixed integer such that $1 \leq j \leq r-1.$ Then for $m_j+2 \leq l \leq q+1,$ $E_{m,j} \text{ is either }  (j-1)q+l-1 \text{ or } (j-1)q+l \text{ for all } k\big(r(l-2)+1\big)+1 \leq m \leq k\big(r(l-1)-j+1\big).$ 	
	\end{corollary}
	\begin{proof}
		By \cref{lemma3.4} we have for $m_j+1 \leq l \leq q+1,$
		$E_{m,j}=(j-1)q+l$ for all $k\big(r(l-1)-j+1\big)+1 \leq m \leq k\big(r(l-1)+1\big).$ Thus, for $m_j+2 \leq l \leq q+1,$ $E_{m,j} \text{ is either }  (j-1)q+l-1 \text{ or } (j-1)q+l \text{ for all } k\big(r(l-2)+1\big)+1 \leq m \leq k\big(r(l-1)-j+1\big).$ 
	\end{proof}
	
	Now we prove a lemma that we use frequently. Before stating the lemma we recall that for $1 \leq j \leq r-1$ and for $2 \leq l \leq q+1,$ every integer $(j-1)q+l$ appears either in the $j$-th column or in the $j+1$-th column of $\Gamma$ (see \cref{lemma3.1}). Also, by \cref{lemma3.2} we have $E_{m,j+1} \leq jq+1$ for all $1 \leq m \leq k(r-j).$
	\begin{lemma}\label{lemma3.6}
		Let $j$ be a fixed integer such that $1 \leq j \leq r-1$ and $l$ be a positive integer such that $3 \leq l \leq q+1.$ If $n_1, n_2$ are non-negative integers such that $n_1 \leq n_2 \leq k(r-j)$ and $\sum_{t=2}^{l-1}N_{(j-1)q+t,j+1} \geq n_1,$ $\sum_{t=2}^{l}N_{(j-1)q+t,j+1} \leq n_2,$ then $E_{m,j}=(j-1)q+l$ for all $k(r(l-1)+1-j)-n_1+1 \leq m \leq k(rl+1-j)-n_2.$ 
		Further, if $l=2$ and $n_2$ is a non-negative integer such that $n_2 \leq k(r-j)$ and $N_{(j-1)q+2,j+1} \leq n_2,$ then $E_{m,j}=(j-1)q+2$ for all $k(r+1-j)+1 \leq m \leq k(2r+1-j)-n_2.$  
	\end{lemma}
	\begin{proof}
		Let $l$ be a positive integer such that $3 \leq l \leq q+1.$ Since $\sum_{t=2}^{l}N_{(j-1)q+t,j+1} \leq n_2,$ we have $\sum_{t=2}^lN_{(j-1)q+t,j} \geq kr(l-1)-n_2.$ Recall that by \cref{lemma3.2}, we have $E_{k(r-j+1),j} \leq (j-1)q+1.$ Thus, $E_{m,j}\leq (j-1)q+l$ for $m \leq kr(l-1)-n_2+k(r-j+1)=k(rl-j+1)-n_2.$ 
		Hence, it is enough to prove that $E_{m,j} \geq (j-1)q+l$ for $m \geq k(r(l-1)+1-j)-n_1+1.$ Since $\sum_{t=2}^{l-1}N_{(j-1)q+t,j+1} \geq n_1,$ we have $\sum_{t=2}^{l-1}N_{(j-1)q+t,j} \leq kr(l-2)-n_1.$ Note that by \cref{lemma3.2}, $E_{k(r-j+1)+1,j} \geq (j-1)q+2.$ Thus, $E_{m,j}\geq (j-1)q+l$ for $m \geq k(r-j+1)+1+kr(l-2)-n_1=k(r(l-1)-j+1)-n_1+1.$
		
		For $l=2,$ the proof is similar.
	\end{proof} 
	\subsection{Combinatorial result}\label{section7}
	
	In this subsection we prove a combinatorial result which we use to prove \cref{theorem8.1}. 
	
	Let $\mathcal{A}$ be the set of all $T$-invariant standard Young tableaux of shape $n \times r$ such that the corresponding standard monomial does not vanish on $X^{v_{\underline{m}}}_{w_{r,n}}.$ For $1 \leq j \leq r-1,$ let $A_j$  be the set of sequences defined by
	$A_j=\{t_{1,j} \leq t_{2,j} \leq \cdots \leq t_{r-j,j}: m_j+1 \leq t_{m,j} \leq q+1 \text{ for all } 1 \leq m \leq r-j\}.$ For $1 \leq j \leq r-1,$ we fix a sequence $ m_{j}+1 \leq t_{1,j} \leq t_{2,j} \leq \cdots \leq t_{r-j,j} \leq q+1.$
	
	For $1 \leq j \leq r,$ let $B_j^{'}$ and $B_j$ be two multisets defined by
	
	$B_j^{'}=\{\underbrace{(j-1)q+2}_{r}, \underbrace{(j-1)q+3}_{r}, \ldots, \underbrace{jq+1}_{r}\}$ 
	
	and
	
	$B_j=B_{j}^{'}\setminus \{(j-1)q+t_{m,j}: 1 \leq m \leq r-j \}$ 
	
	where $\underbrace{t}_r$ denotes that $t$ appears in the set for $r$ times.
	
	Let $\underline{t}:=(t_{1,j} \leq \cdots \leq t_{r-j,j})_{ 1 \leq j \leq r-1}  \in A_1 \times A_2 \times \cdots \times A_{r-1}.$ Consider the Young tableau $\Gamma_{\underline{t}}$ with $E^{\Gamma_{\underline{t}}}_{m,j}$ as its $(m,j)$-th entry and $row_m$ as its $m$-th row defined as follows: 
	\begin{align}E^{\Gamma_{\underline{t}}}_{m,1}=
		\left\{\begin{array}{lr}
			1 & \text{if $1 \leq m \leq r$}\\
			min\{B_1\setminus \{E^{\Gamma_{\underline{t}}}_{l,1}: r+1 \leq l \leq m-1\}\} & \text{if $r+1 \leq m \leq rq+1$}
		\end{array}\right.;
		\label{(4.1)}
	\end{align}
	for $2 \leq j \leq r-1,$ 
	\begin{align}
		E^{\Gamma_{\underline{t}}}_{m,j}=
		\left\{\begin{array}{lr}
			(j-2)q+t_{m,j-1} & \text{if $1 \leq m \leq r-j+1$ }\\
			min\{B_j\setminus \{E^{\Gamma_{\underline{t}}}_{l,j}: r-j+2 \leq l \leq m-1\}\} & \text{if $r-j+2 \leq m \leq rq+1$} 
		\end{array}\right.;
		\label{(4.2)} 
	\end{align} and
	\begin{align}
		E^{\Gamma_{\underline{t}}}_{m,r}=
		\left\{\begin{array}{lr}
			(r-2)q+t_{1,r-1} & \text{if $m=1$}\\
			(r-1)q+l &  \begin{matrix}
				\text{where } l  \text{ is the least positive integer such }\\
				\text{ that } (l-1)r+1 \geq m.
			\end{matrix}
		\end{array}\right.
		\label{(4.3)}
	\end{align}
	
	\underline{Claim:} We claim that $\Gamma_{\underline{t}}$ is $T$-invariant.
	
	Let $N^{\Gamma_{\underline{t}}}_{t,j}$ be the number of boxes in the $j$-th column of $\Gamma_{\underline{t}}$  containing $t.$ By \eqref{(4.1)} we have $N^{\Gamma_{\underline{t}}}_{1,1}=r.$ Note that every element of $B_{j}^{'}$ ($1 \leq j \leq r-1$) is either appearing in $j$-th column or $j+1$-th column. Thus, for  $1 \leq j \leq r-1,$ $N^{\Gamma_{\underline{t}}}_{t,j}+N^{\Gamma_{\underline{t}}}_{t,j+1}=r$ for all $(j-1)q+2 \leq t \leq jq+1.$ By \eqref{(4.3)} $N^{\Gamma_{\underline{t}}}_{t,r}=r$ for all $(r-1)q+2 \leq t \leq rq+1.$ Hence, the claim is proved.
	
	\underline{Claim:} We claim that $\Gamma_{\underline{t}}$ is a standard Young tableau.
	
	Assume that $1 \leq j \leq r-1.$ If $1 \leq m \leq r-j,$ then $E^{\Gamma_{\underline{t}}}_{m,j+1}-E^{\Gamma_{\underline{t}}}_{m,j}=\big( (j-1)q+t_{m,j}\big) - \big( (j-2)q+t_{m,j-1}\big)=q+(t_{m,j}-t_{m,j-1}).$ Since $t_{m,j}-t_{m,j-1} \geq m_j+1-q-1=m_j-q \geq 1-q,$ we have $E^{\Gamma_{\underline{t}}}_{m,j+1}-E^{\Gamma_{\underline{t}}}_{m,j} \geq q+(1-q) = 1.$  If $r-j+1 \leq m \leq rq+1,$ then $E^{\Gamma_{\underline{t}}}_{m,j+1} \geq jq+2$ and  $E^{\Gamma_{\underline{t}}}_{m,j} \leq jq+1.$ Thus, $E^{\Gamma_{\underline{t}}}_{m,j+1}-E^{\Gamma_{\underline{t}}}_{m,j} \geq 1.$ Hence, entries in each row is strictly increasing from left to right. By definition entries in each column is non-decreasing from top to bottom.
	
	\underline{Claim}: We claim that the standard monomial $p_{\Gamma_{\underline{t}}}$ associated to $\Gamma_{\underline{t}}$ does not vanish on $X^{v_{\underline{m}}}_{w_{r,n}}.$ 
	
	Note that by \eqref{(4.1)} we have $E^{\Gamma_{\underline{t}}}_{1,1}=1.$ By \eqref{(4.2)} and \eqref{(4.3)} we have $E^{\Gamma_{\underline{t}}}_{1,j}=(j-2)q+t_{1,j-1} \geq (j-2)q+m_{j-1}+1$ for all $2 \leq j \leq r.$ Hence, $row_1 \geq v_{\underline{m}}.$ By \eqref{(4.1)} and \eqref{(4.2)} for all $1 \leq j \leq r-1,$ $E^{\Gamma_{\underline{t}}}_{n,j} \in B_j$ and by \eqref{(4.3)} we have $E^{\Gamma_{\underline{t}}}_{n,r} \leq rq+1.$ Thus, $E^{\Gamma_{\underline{t}}}_{n,j} \leq jq+1$ for all $1 \leq j \leq r.$ Hence, $row_n \leq w_{r,n}.$ Therefore, the claim is proved.   
	
	Consider the map \begin{center}$f: A_1 \times A_2 \times \cdots \times A_{r-1} \longrightarrow  \mathcal{A}$\end{center} defined by $f(\underline{t})=\Gamma_{\underline{t}},$ where $\underline{t}:=(t_{1,j} \leq \cdots \leq t_{r-j,j})_{ 1 \leq j \leq r-1}.$ 
	
	\begin{lemma}\label{lemma4.1*}
		$f$ is bijective.
	\end{lemma}
	
	\begin{proof}
		
		First we prove that $f$ is injective. Let $\underline{t}$ and $\underline{t'}$ be two elements in $A_1 \times A_2 \times \cdots \times A_{r-1}$ such that $f(\underline{t})=f(\underline{t'}).$ Thus, $\Gamma_{\underline{t}}=\Gamma_{\underline{t'}}.$ Hence,   the $(i,j)$-th entry of $\Gamma_{\underline{t}}$ is equal to the $(i,j)$-th entry of $\Gamma_{\underline{t'}}$ for all $1 \leq i \leq r-j+1$ and $2 \leq j \leq r.$ Therefore, $t_{i,j-1}=t'_{i,j-1}$ for all $1 \leq i \leq r-j+1$ and $2 \leq j \leq r.$ Hence, $\underline{t}=\underline{t'}.$
		
		Now we prove that $f$ is surjective. Let $\Gamma \in \mathcal{A}.$ Thus, by \cref{lemma3.2}, for all $1 \leq j \leq r-1$  there exists a sequence $m_j+1 \leq t_{1,j} \leq t_{2,j} \leq \cdots \leq t_{r-j,j}$ such that \begin{align*}E_{m,j+1}=(j-1)q+t_{m,j} \text{ for all } 1 \leq m \leq r-j. 
		\end{align*} 
		We subclaim that for all $1 \leq j \leq r-1,$ every integer in $B_j'\setminus\{E_{m,j+1}: 1 \leq m \leq r-j\}$ appear as $E_{m,j}$ for all $r-j+2 \leq m \leq rq+1.$ By \cref{lemma3.1}, we have every integer between $(j-1)q+2$ and $jq+1$ appears either in the $j$-th column or in the $j+1$-th column of $\Gamma$. Also, by \cref{lemma3.2}, we have $E_{r-j+1,j+1} \geq jq+2.$ Thus, every integer in $B_j'\setminus\{E_{m,j+1}: 1 \leq m \leq r-j\}$ appear as $E_{m,j}$ for all $r-j+2 \leq m \leq rq+1.$ Further, by \cref{corollary3.3}, we have $E_{m,r}=(r-1)q+l,$ where $l$ is the least positive integer such that $(l-1)r+1 \geq m.$ Therefore, $\Gamma=\Gamma_{\underline{t}}$ for some  $\underline{t}=(t_{1,j} \leq t_{2,j} \leq \cdots \leq t_{r-j,j})_{1 \leq j \leq r-1} \in A_1 \times A_2 \times \cdots \times A_{r-1}.$ 
	\end{proof}
	
	\subsection{$R_1$-generation}\label{section4}
	In this subsection we prove that the homogeneous coordinate ring of $T\backslash \backslash (X^{v_{\underline{m}}}_{w_{r,n}})^{ss}_T(\mathcal{L})$ is generated by $H^0(X^{v_{\underline{m}}}_{w_{r,n}}, \mathcal{L})^{T}$ as a $\mathbb{C}$-algebra. 
	
	Let $\tilde{X}= T \backslash \backslash (X^{v_{\underline{m}}}_{w_{r,n}})^{ss}_T(\mathcal{L}).$
	Then we have $\tilde{X}= Proj(R),$ where $R=\bigoplus\limits_{k\in\mathbb{Z}_{\geq 0}}R_k.$ Note that $R_{k}$'s  are finite dimensional vector spaces.
	\begin{theorem}\label{lemma4.1}
		$R$ is generated by $R_1$ as a $\mathbb{C}$-algebra.
	\end{theorem}
	
	\begin{proof}  Let $f \in R_k$ be a standard monomial and $\Gamma$ be the standard Young tableau associated to $f.$ We claim that $f=f_1\cdot f_2$ for some $f_1$ in $R_1$ and $f_2$ in $R_{k-1}.$ Now it is enough to prove that $\Gamma$ has a subtableau $\Gamma'$ of shape $n \times r$ which is $T$-invariant. Consider the sub-tableau $\Gamma'$ of $\Gamma$ with the following rows: \begin{equation*}row_m(\Gamma')=
			\left\{\begin{array}{lr}
				row_{k(m-1)+1}(\Gamma) & \text{for $1 \leq m \leq r$}\\
				row_{km}(\Gamma) & \text{for $r+1 \leq m \leq n$}\\
			\end{array}\right.
		\end{equation*}
		where $row_{m}(\Gamma)$ is the $m$-th row of $\Gamma.$
		
		Let $E_{i,j}$  be the $(i,j)$-th entry of the tableau $\Gamma$ and $N_{t,j}$ be the number of boxes in the $j$-th column of $\Gamma$  containing $t.$ 
		
		Assume that $1 \leq j \leq r-1.$ Let $1 \leq a_{1,j} \leq r-j$ be the largest positive integer such that $$E_{1,j+1}=E_{k+1,j+1}=\cdots=E_{k(a_{1,j}-1)+1,j+1}.$$ Then by \cref{lemma3.2}, we have $E_{k(a_{1,j}-1)+1,j+1}=(j-1)q+t_{a_{1,j}}$ for some fixed integer $t_{a_{1,j}}$ such that $m_j+1 \leq t_{a_{1,j}} \leq q+1.$ 
		
		Let $a_{1,j} < a_{2,j} \leq r-j$ be the largest positive integer such that $E_{ka_{1,j}+1,j+1}=\cdots=E_{k(a_{2,j}-1)+1,j+1}.$ Then by \cref{lemma3.2}, we have $E_{k(a_{2,j}-1)+1,j+1}=(j-1)q+t_{a_{2,j}}$ for some fixed integer $t_{a_{2,j}}$ such that $t_{a_{1,j}}+1 \leq t_{a_{2,j}} \leq q+1.$ 
		
		Now proceeding in this way we get a sequence of positive integers $1 \leq a_{1,j} < a_{2,j} < \cdots < a_{z_j,j}=r-j$ such that $E_{ka_{i-1,j}+1,j+1}=\cdots=E_{k(a_{i,j}-1)+1,j+1}$ for all $2 \leq i \leq z_j.$ Then by \cref{lemma3.2}, we have $E_{k(a_{i,j}-1)+1,j+1}=(j-1)q+t_{a_{i,j}}$ for some fixed integer $t_{a_{i,j}}$ such that $t_{a_{i-1,j}}+1 \leq t_{a_{i,j}} \leq q+1.$ 
		
		Now we prove that for $1 \leq j \leq r-1$ and $2 \leq l \leq q+1,$ every integer $(j-1)q+l$ appears exactly $r$ times in $\Gamma'.$ We prove this case by case. 
		\begin{itemize}
			\item[Case (i):] Assume that $l=2.$ 
			\begin{itemize}
				\item[Subcase (1):] If $t_{a_{1,j}} =2,$ then we prove that $(j-1)q+2$ appears $r-a_{1,j}$ times in the $j$-th column and $a_{1,j}$ times in the $j+1$-th column of $\Gamma'.$  
				\item[Subcase (2):] If $t_{a_{1,j}} >2,$ then we prove that $(j-1)q+2$ appears exactly $r$ times in the $j$-th column of $\Gamma'$.
			\end{itemize}
			
			\item[Case (ii):] Assume that $1 \leq i \leq z_j$ and $t_{a_{i-1,j}}+1\leq l \leq t_{a_{i,j}}-1$ such that $l \neq 2.$ Then $(j-1)q+l$ appears exactly $r$ times in the $j$-th column of $\Gamma',$ where we define $a_{0,j}=0$ and $t_{a_{0,j}}=1$. 
			\item[Case (iii):]  Assume that $1 \leq i \leq z_j$ and $l = t_{a_{i,j}}(\neq 2).$ Then we prove that $(j-1)q+t_{a_{i,j}}$ appears $r-(a_{i,j}-a_{i-1,j})$ times in the $j$-th column and $a_{i,j}-a_{i-1,j}$ times in the $j+1$-th column of $\Gamma'.$ 
			\item[Case (iv):] Assume that $t_{a_{z_{j},j}}+1 \leq l\leq q+1.$ Then we prove that $(j-1)q+l$ appears exactly $r$ times in the $j$-th column of $\Gamma'$. 
		\end{itemize}
		Now we prove in Case (i).
		
		Subcase (1): Assume that $t_{a_{1,j}}=2.$
		
		Since $t_{a_{1,j}} \geq m_j+1 \geq 2,$ we have $m_j=1.$ Thus, by \cref{lemma3.4}(ii) we have $E_{u,j}=(j-1)q+2$ for all $k(r-j+1)+1 \leq u \leq k(r+1).$ Hence, $E_{k(m-1)+1,j}=(j-1)q+2  \text{ for all } r-j+2 \leq m \leq r.$
		
		By the above paragraph, we have $E_{k(r+1),j}=(j-1)q+2.$ Now it is enough to prove that $E_{u,j} \leq (j-1)q+2$ for $u \leq k(2r-a_{1,j}-j+1).$  Since $E_{ka_{1,j}+1,j+1}=(j-1)q+t_{a_{2,j}},$ we have $N_{(j-1)q+2,j+1} \leq ka_{1,j}.$ Thus, $N_{(j-1)q+2,j} \geq k(r-a_{1,j}).$ Since $E_{k(r-j+1),j} \leq (j-1)q+1$ (see \cref{lemma3.2}), we have $E_{u,j}\leq (j-1)q+2$ for $u \leq k(r-j+1)+k(r-a_{1,j})=k(2r-j-a_{1,j}+1).$ Hence, $E_{km,j}=(j-1)q+2  \text{ for all }  r+1 \leq m \leq 2r-a_{1,j}-j+1.$
		
		Recall that $E_{k(m-1)+1,j+1}=(j-1)q+t_{a_{1,j}}$ for all $1 \leq m \leq a_{1,j}.$
		
		Therefore, Subcase (1) follows from above three paragraphs.
		
		Subcase (2): Assume that $t_{a_{1,j}}>2.$
		
		Since by \cref{lemma3.2}(ii) $E_{k(r-j+1)+1,j} \geq (j-1)q+2$, it is enough to prove that $E_{u,j} \leq (j-1)q+2$ for all $u \leq k(2r-j+1).$ Since $t_{a_{1,j}}>2,$ we have $(j-1)q+2$ appears $kr$ times in the $j$-th column. Since by \cref{lemma3.2}(i) $E_{k(r-j+1),j} \leq (j-1)q+1$, we have $E_{u,j}\leq(j-1)q+2$ for $u \leq k(r-j+1)+kr=k(2r-j+1).$ Therefore, $E_{k(m-1)+1,j}=(j-1)q+2  \text{ for all } r-j+2 \leq m \leq r$ and $E_{km,j}=(j-1)q+2  \text{ for all }  r+1 \leq m \leq 2r-j+1.$ Therefore, Subcase (2) is proved.
		
		Now we prove in Case (ii).
		
		Assume that $1 \leq i \leq z_j$ and $t_{a_{i-1,j}}+1\leq l \leq t_{a_{i,j}}-1$ such that $l \neq 2.$ 
		Since $E_{ka_{i-1,j}+1,j+1}=(j-1)q+t_{a_{i,j}}$, we have $\sum_{t=2}^lN_{(j-1)q+t,j+1} \leq ka_{i-1,j}.$ Also, since $E_{k(a_{i-1,j}-1)+1,j+1}=(j-1)q+t_{a_{i-1,j}},$ we have $\sum_{t=2}^{l-1}N_{(j-1)q+t,j+1} \geq k(a_{i-1,j}-1)+1.$ Now consider $n_1=k(a_{i-1,j}-1)+1$ and $n_2=ka_{i-1,j}.$ Then by using \cref{lemma3.6}, we have $$E_{u,j}=(j-1)q+l \text{ for all } k(r(l-1)-a_{i-1,j}-j+2) \leq u \leq k(rl-a_{i-1,j}-j+1).$$ Therefore, Case (ii) is proved.
		
		Now we prove in Case (iii). Assume that $l=t_{a_{i,j}} (\neq 2).$ 
		
		Since $E_{ka_{i,j}+1,j+1}=(j-1)q+t_{a_{i+1,j}}$, we have $\sum_{t=2}^{t_{a_{i,j}}}N_{(j-1)q+t,j+1} \leq ka_{i,j}.$ Also, since $E_{k(a_{i-1,j}-1)+1,j+1}=(j-1)q+t_{a_{i-1,j}}$, we have $\sum_{t=2}^{t_{a_{i,j}}-1}N_{(j-1)q+t,j+1} \geq k(a_{i-1,j}-1)+1.$ Now consider $n_1=k(a_{i-1,j}-1)+1$ and $n_2=ka_{i,j}.$ Then by using \cref{lemma3.6}, we have $$E_{u,j}=(j-1)q+t_{a_{i,j}} \text{ for all } k(r(t_{a_{i,j}}-1)-a_{i-1,j}-j+2) \leq u \leq k(rt_{a_{i,j}}-a_{i,j}-j+1).$$ Hence, $E_{km,j}=(j-1)q+t_{a_{i,j}} \text{ for all } r(t_{a_{i,j}}-1)-a_{i-1,j}-j+2 \leq u \leq rt_{a_{i,j}}-a_{i,j}-j+1.$ 
		
		Note that  $E_{k(m-1)+1,j+1}=(j-1)q+t_{a_{i,j}}$ for all $a_{i-1,j}+1 \leq m \leq a_{i,j}.$ 
		
		Now we prove in Case (iv).
		
		Assume that $t_{a_{z_{j},j}}+1 \leq l\leq q+1.$
		By \cref{lemma3.2}, we have $\sum_{t=2}^lN_{(j-1)q+t,j+1} \leq k(r-j).$ 
		Since $E_{k(r-j-1)+1,j+1}=(j-1)q+t_{a_{z_j,j}},$ we have $\sum_{t=2}^{l-1}N_{(j-1)q+t,j+1} \geq k(r-j-1)+1.$ Now consider $n_1=k(r-j-1)+1$ and $n_2=k(r-j).$ Then by using \cref{lemma3.6}, we have $$E_{u,j}=(j-1)q+l \text{ for all } k(r(l-2)+2) \leq u \leq k(r(l-1)+1).$$ Hence, $E_{km,j}=(j-1)q+l \text{ for all } r(l-2)+2 \leq m \leq r(l-1)+1.$ Therefore, Case (iv) is proved.
		
		Now we prove that every integer in $\{(r-1)q+2,(r-1)q+3, \ldots, rq+1\}$ appears $r$ times in the $r$-th column of $\Gamma'.$ By \cref{corollary3.3}, we have $E_{k(m-1)+1,r}=(r-1)q+2$ for all $2 \leq m \leq r$ and $E_{k(r+1),r}=(r-1)q+2.$ For $3 \leq l \leq q+1,$ $E_{km,r}=(r-1)q+l$ for all $r(l-2)+2 \leq m \leq r(l-1)+1.$ Therefore, for $2 \leq l \leq q+1,$ every integer $(r-1)q+l$ appears $r$ times in the $r$-th column of $\Gamma'.$ 
	\end{proof}
	
	\section{Restriction of the $T$-invariant sections on the Deodhar component}\label{section5}
	
	In this section we study the Deodhar component $\mathcal{R}^{v_{\underline{m}}}_{w_{r,n}}$ in the Richardson variety $X^{v_{\underline{m}}}_{w_{r,n}}.$ Also, we find the common factor of the restriction of all $T$-invariant standard monomials to $\mathcal{R}^{v_{\underline{m}}}_{w_{r,n}}$ and the non-common factor of the restriction of a $T$-invariant standard monomial to $\mathcal{R}^{v_{\underline{m}}}_{w_{r,n}}.$ 
	
	In this section we use the terminology ''common factor" of a set of monomials and "non-common factor" of a monomial in a set of monomials. A "common factor" of a set of monomials is the greatest common divisor of the monomials in the set and the "non-common factor" of a monomial in the set is the resultant monomial divided by the common factor. Note that the common factor of the monomials in a set is unique.
	
	Let $U^{*}_{-\alpha_{i}}=U_{-\alpha_{i}} \setminus \{1\}$ and let $c_{i,j}$ be the coordinate function on $j$-th appearance of $U^{*}_{-\alpha_{i}}$. To state \cref{lemma5.1} first we introduce an $n \times r$ matrix. Let $e_{l,j}$ be the $(l,j)$-th entry of the matrix. Let the first column of the matrix be  
	\begin{equation}e_{l,1}=
		\left\{\begin{array}{lr}
			1 & \text{for $l=1$}\\
			\prod_{t=1}^{l-1}c_{t,1} & \text{for $2 \leq l \leq q+1$}\\
			0 & \text{for $q+2 \leq l \leq qr+1$}
		\end{array}\right..
		\label{(5.1)}
	\end{equation} 
	Let the $j$-th ($2 \leq j \leq r$) column of the matrix be
	\small{\begin{equation}
			e_{l,j}=\left\{\begin{array}{lr}
				0 & \text{for $1 \leq l \leq (j-2)q+m_{j-1}$}\\
				1 & \text{for $l=(j-2)q+m_{j-1}+1$}\\
				\begin{matrix}\displaystyle\sum_{\small{
							t=(j-2)q
							+m_{j-1}+1}}^{l-1}\bigg(\big(\prod_{\small{
							m=(j-2)q+m_{j-1}+1}}^{t-1}c_{m,2}\big)\big(\prod_{m=t}^{l-1}c_{m,1}\big)\bigg)\\+\displaystyle\prod_{\small{
							t=(j-2)q+m_{j-1}+1}}^{l-1}c_{t,2} \end{matrix} & \text{for $\begin{matrix}(j-2)q+m_{j-1}+2 \leq l\\
						\leq (j-1)q+1 \end{matrix}$}\\
				\displaystyle \big(\prod_{t=(j-2)q+m_{j-1}+1}^{(j-1)q}c_{t,2}\big)\big(\prod_{t=(j-1)q+1}^{l-1}c_{t,1}\big) & \text{for $\begin{matrix}(j-1)q+2 \leq l  \leq jq+1\end{matrix}$}\\
				0 & \text{for $jq+2 \leq l \leq rq+1$}
			\end{array}\right..
			\label{(5.2)}
	\end{equation}}
	
	\normalsize{}	\begin{lemma}\label{lemma5.1}
	The first $r$ columns of matrices $u$ for which  $uP^{\alpha_r}/P^{\alpha_r}$ belong to the Deodhar component $\mathcal{R}^{v_{\underline{m}}}_{w_{r,n}}$ are precisely the ones that are described in \ref{(5.1)} and \ref{(5.2)}. 
	\end{lemma}
	\begin{proof}
		Consider the Deodhar component of $X^{v_{\underline{m}}}_{w_{r,n}}$ corresponding to the subexpression $v_{\underline{m}}\\=\underbrace{1 \ldots 1}_{q}\prod_{j=2}^{r} (\underbrace{1 \ldots 1}_{2q-m_{j-1}}s_{(j-2)q+m_{j-1}}\cdots s_{j+1}s_j).$
		The Deodhar component $\mathcal{R}^{v_{\underline{m}}}_{w_{r,n}}$ in $X^{v_{\underline{m}}}_{w_{r,n}}$ is\\ $$\bigg(\prod_{i=q}^{1}U^{*}_{-\alpha_{i}}\bigg)\bigg(\prod_{j=2}^{r}\big(\prod_{i=jq}^{(j-2)q+m_{j-1}+1}U^{*}_{-\alpha_{i}}\big)\bigg)v_{\underline{m}}P^{\alpha_r}/P^{\alpha_r}.$$ 
		
		Since the one line notation of $v_{\underline{m}}$ is $(1, m_1+1, q+m_2+1, \ldots, (r-2)q+m_{r-1}+1)$, we have to consider first column, $m_1+1$-th column, $q+m_2+1$-th column, $\ldots,$ $(r-2)q+m_{r-1}+1$-th column for computing restriction of the Pl\"{u}cker coordinates to $\mathcal{R}^{v_{\underline{m}}}_{w_{r,n}}$. Therefore, $n \times r$ matrices formed by these $r$ columns in $\mathcal{R}^{v_{\underline{m}}}_{w_{r,n}}$ are of the following form. For simplicity of notation we consider column $(j-2)q+m_{j-1}+1$ as column $j$ for all $2 \leq j \leq r.$ 
		
		Since the product $\prod_{j=2}^{r}\big(\prod_{i=jq}^{(j-2)q+m_{j-1}+1}U^{*}_{-\alpha_{i}}\big)$ has $(1,1)$-th entry $1$ and $(l,1)$-th ($l \geq 2$) entry zero, we have $(l,1)$-th $(2 \leq l \leq q+1)$ entry of $\big(\prod_{i=q}^{1}U^{*}_{-\alpha_i}\big)\big(\prod_{j=2}^{r}(\prod_{i=jq}^{(j-2)q+m_{j-1}+1}U^{*}_{-\alpha_{i}})\big)$ is equal to the $(l,1)$-th $(2 \leq l \leq q+1)$ entry of the product $\prod_{i=q}^{1}U^{*}_{-\alpha_i}.$ Hence, the first column of the matrix is as given in \ref{(5.1)}. 
		
		Let $2 \leq j \leq r.$ Since the product $\prod_{m=j+1}^{r}\big(\prod_{i=mq}^{(m-2)q+m_{l-1}+1}U^{*}_{-\alpha_{i}}\big)$ has $((j-2)q+m_{j-1}+1,(j-2)q+m_{j-1}+1)$-th entry $1$ and $(l,(j-2)q+m_{j-1}+1)$-th ($l \leq (j-2)q+m_{j-1}$ and $l \geq (j-2)q+m_{j-1}+2$) entry zero, we have $(l,(j-2)q+m_{j-1}+1)$-th entry $((j-2)q+m_{j-1}+2 \leq l \leq jq+1)$ of $\big(\prod_{i=q}^{1}U^{*}_{-\alpha_i}\big)\big(\prod_{j=2}^{r}(\prod_{i=jq}^{(j-2)q+m_{j-1}+1}U^{*}_{-\alpha_{i}})\big)$ is equal to the $(l,(j-2)q+m_{j-1}+1)$-th entry $((j-2)q+m_{j-1}+2 \leq l \leq jq+1)$ of the product $\big(\prod_{i=(j-1)q}^{(j-2)q+m_{j-1}+1}U^{*}_{-\alpha_i}\big)\big(\prod_{i=jq}^{(j-2)q+m_{j-1}+1}U^{*}_{-\alpha_i}\big).$ Hence, the $j$-th column of the matrix is as given in \ref{(5.2)}.
	\end{proof}
	
	\begin{example}
		Let $(n,r,q)=(10,3,3).$ Consider $w_{3,10}=(s_3s_2s_1)(s_6s_5s_4s_3s_2)(s_9s_8s_7s_6s_5s_4s_3).$ Take $v_{(2,2)}=(s_2)(s_5s_4s_3).$ Then $v_{(2,2)}$ as a subexpression in $w_{3,10}$ is $(1.1.1)(1.1.1.1.s_2)(1.1.1.1.s_5\\s_4s_3).$ Then the  Deodhar component $\mathcal{R}^{v_{(2,2)}}_{w_{3,10}}$ is $U^{*}_{-\alpha_3}U^{*}_{-\alpha_2}U^{*}_{-\alpha_1}U^{*}_{-\alpha_6}U^{*}_{-\alpha_5} U^{*}_{-\alpha_4}U^{*}_{-\alpha_3}U^{*}_{-\alpha_9}U^{*}_{-\alpha_8}U^{*}_{-\alpha_7}\\U^{*}_{-\alpha_6}(s_2)(s_5s_4s_3)P^{\alpha_3}/P^{\alpha_3}.$
		
		Since the one line notation of $v_{(2,2)}$ is $(1,3,6),$ we have to consider the first, third and sixth column for computing restriction of the Pl\"{u}cker coordinates to $\mathcal{R}^{v_{(2,2)}}_{w_{3,10}}$. Therefore $10 \times 3$ matrices formed by these $3$ columns in $\mathcal{R}^{v_{(2,2)}}_{w_{3,10}}$ are of the following form: 
		$$
		\begin{pmatrix}
			1 & 0 & 0 \\
			c_{11} & 0 & 0\\
			c_{11}c_{21} & 1 & 0\\
			c_{11}c_{21}c_{31} & c_{31}+c_{32} & 0\\
			0 & c_{32}c_{41} & 0\\
			0 & c_{32}c_{41}c_{51} & 1\\
			0 & c_{32}c_{41}c_{51}c_{61} & c_{61}+c_{62}\\
			0 & 0 & c_{62}c_{71}\\
			0 & 0 & c_{62}c_{71}c_{81}\\
			0 & 0 & c_{62}c_{71}c_{81}c_{91}
		\end{pmatrix}.
		$$
	\end{example}
	\begin{lemma}\label{lemma5.3}
		Let $p_{i_1,i_2\ldots,i_r}$ be a Pl\"{u}cker coordinate dividing a $T$-invariant standard monomial. Let $M$ be a matrix in the Deodhar component $\mathcal{R}^{v_{\underline{m}}}_{w_{r,n}}.$ If  $i_{j-1}$ and $i_{j}$ do not simultaneously lie in the interval $[(j-2)q+m_{j-1}+1,(j-1)q+1]$ for all $2 \leq j \leq r,$ then $p_{i_1,i_2,\ldots,i_r}(M)$ is the product of diagonals.  
	\end{lemma}
	\begin{proof}
		Let $e_{k,j}$ be the $(k,j)$-th entry of the matrix $M.$ 
		Since $p_{i_1,i_2,\ldots, i_r}(M)$ is the determinant of the submatrix of $M$ with $i_1$-th row, $i_2$-th row, $\ldots,$ $i_r$-th row, we have $$p_{i_1,i_2,\ldots, i_r}(M)=\begin{vmatrix}
			e_{i_1,1} & e_{i_1,2} & \cdots & e_{i_1,r}\\
			e_{i_2,1} & e_{i_2,2} & \cdots & e_{i_2,r}\\
			\cdot & \cdot & \cdots & \cdot\\
			e_{i_r,1} & e_{i_r,2} & \cdots & e_{i_r,r}\\
		\end{vmatrix}.$$ 
		Since the one line notation of $v_{\underline{m}}$ and $w_{r,n}$ are $(1, m_1+1, q+m_2+1, \ldots, (r-2)q+m_{r-1}+1)$ and $(q+1, 2q+1, 3q+1, \ldots, rq+1)$ respectively, we have $p_{i_1,i_2,\ldots, i_r}(M) \neq 0$ if and only if 
		$1 \leq i_1 \leq q+1$ and
		$(l-2)q+m_{l-1}+1 \leq i_l \leq lq+1$ for all $2 \leq l \leq r.$
		(see \cite[Lemma 9, p.667]{BL}). Also by \eqref{(5.1)} and \eqref{(5.2)} we have for $1 \leq l \leq r$  \begin{equation}\label{5.4}e_{i_l,j} = 0 \text{ if } 
			\left\{\begin{array}{lr}
				i_l \geq q+2 & \text{for $j=1$}\\
				\text{either } i_l  \leq (j-2)q+m_{j-1} \text{ or } i_l \geq jq+2 & \text{ for $2 \leq j \leq r$}
			\end{array}
			\right..
		\end{equation}
		Since for $3 \leq l \leq r$ we have $i_l \geq (l-2)q+m_{l-1}+1 \geq q+2.$ Thus, by \eqref{5.4} $e_{i_l,1}=0$ for all $3 \leq l \leq r.$ Further, for $4 \leq l \leq r,$ $i_l \geq (l-2)q+m_{l-1}+1 \geq 2q+2.$ Thus, $e_{i_l,2}=0$ for all $ 4 \leq l \leq r.$
		Consider $3 \leq j \leq r.$ For $l \leq j-2$ we have $i_l \leq lq+1 \leq (j-2)q+1.$ Thus by \eqref{5.4} we have $e_{i_l,j}=0$ for all $l \leq j-2$. For $l \geq j+2$ we have $i_l \geq (l-2)q+m_{l-1}+1 \geq jq+m_{j+1}+1 \geq jq+2.$ Thus by \eqref{5.4} we have $e_{i_l,j}=0$ for all $l \geq j+2$.
		
		Also by the given hypothesis for $2 \leq j \leq r,$ $i_{j-1}$ and $i_{j}$ do not simultaneously lie in the interval $[(j-2)q+m_{j-1}+1,(j-1)q+1].$ Thus, we have either $i_{j-1} \leq (j-2)q+m_{j-1}$ or $i_j \geq (j-1)q+2.$ Therefore, by \eqref{5.4} we have either $e_{i_{j-1},j}=0$ or $e_{i_j,j-1}=0.$ 
		
		Therefore, $p_{i_1,i_2,\ldots,i_r}(M)$ is the determinant of the following $r \times r$ submatrix $A=$ 
		$$\begin{pmatrix}
			e_{i_1,1}&e_{i_1,2}&0&\cdots&0&0&0&\cdots&0\\
			e_{i_2,1}&e_{i_2,2}&e_{i_2,3}&\cdots&0&0&0&\cdots&0\\
			0&e_{i_3,2}&e_{i_3,3}&\cdots&0&0&0&\cdots&0\\
			\vdots&\vdots&\vdots&\cdots&\vdots&\vdots&
			\vdots&\cdots&\vdots\\
			0&0&0&\cdots&e_{i_{j-1},j-1}&e_{i_{j-1},j}&0&\cdots&0\\
			0&0&0&\cdots&e_{i_{j},j-1}&e_{i_{j},j}&e_{i_{j},j+1}&\cdots&0\\
			0&0&0&\cdots&0&e_{i_{j+1},j}&e_{i_{j+1},j+1}&\cdots&0\\
			\vdots&\vdots&\vdots&\cdots&\vdots&\vdots&\vdots&
			\cdots&\vdots\\
			0&0&0&\cdots&0&0&0&\cdots&e_{i_{r},r}
		\end{pmatrix}$$ \\where for all $2 \leq j \leq r$ either $e_{i_{j-1},j}$ is zero or $e_{i_j,j-1}$ is zero. Note that $p_{i_1,i_2,\ldots, i_r}(M)=e_{1,1}\cdot A_{1,1}-e_{1,2}\cdot A_{1,2},$ where $A_{1,1}$ (respectively, $A_{1,2}$) is the minor corresponding to $e_{i_1,1}$ (respectively, $e_{i_1,2}$). If $e_{i_1,2}=0,$ then $p_{i_1,i_2,\ldots, i_r}(M)=e_{i_1,1}\cdot A_{1,1}.$ If $e_{i_1,2} \neq 0,$ then $e_{i_2,1}=0.$ Thus, all the entries in the first column of $A_{1,2}$ is zero. Hence, in both cases $p_{i_1,i_2,\ldots, i_r}(M)=e_{i_1,1}\cdot A_{1,1}.$ Note that the $(r-1) \times (r-1)$ submatrix $A_{1,1}$ of $A$ has the same form as $A.$ Hence, by induction $p_{i_1,i_2,\ldots, i_r}(M)$ is the product of diagonal entries.		
	\end{proof}
	\begin{lemma}\label{lemma5.4}
		Let $p_{i_1,i_2\ldots,i_r}$ be a Pl\"{u}cker coordinate dividing a $T$-invariant standard monomial.  Then $i_{j-1}$ and $i_{j}$ do not simultaneously lie in the interval $[(j-2)q+m_{j-1}+1,(j-1)q+1]$ for all $2 \leq j \leq r$.  
	\end{lemma}
	\begin{proof}
		Assume that $2 \leq j \leq r-1.$ 
		By \cref{lemma3.2}, we have $E_{r-j+1,j} \leq (j-1)q+1$ and $E_{m,j} \geq (j-1)q+2$ for all $m \geq r-j+2.$ 
		Also, by \cref{lemma3.4}, we have $E_{m,j-1} \leq (j-2)q+m_{j-1}$ for all $m \leq rm_{j-1}-j+2$ and $E_{m,j-1} \geq (j-2)q+m_{j-1}+1$ for all $m \geq rm_{j-1}-j+3.$ Thus, if $(j-2)q+m_{j-1}+1 \leq i_{j-1} < i_j \leq (j-1)q+1,$ then $i_{j-1}=E_{l,j-1}$ for some $l \geq rm_{j-1}-j+3$ and $i_j=E_{k,j}$ for some $k \leq r-j+1,$ which is a contradiction to the fact that $i_{j-1}$ and $i_{j}$ are in the same row.
		
		Now we assume that $j=r.$ By \cref{lemma3.2}, we have $E_{1,r} \leq (r-1)q+1$ and $E_{i,r} \geq (r-1)q+2$ for all $i \geq 2.$ By \cref{lemma3.4}, we have $E_{m,r-1} \leq (r-2)q+m_{r-1}$ for all $m \leq rm_{r-1}-r+2$ and $E_{m,r-1} \geq (r-2)q+m_{r-1}+1$ for all $m \geq rm_{r-1}-r+3.$ Thus, if $(r-2)q+m_{r-1}+1 \leq i_{r-1} < i_r \leq (r-1)q+1,$ then $i_{r-1}=E_{l,r-1}$ for some $l \geq rm_{r-1}-r+3$ and $i_r=E_{k,r}$ for some $k \leq 1,$ which is a contradiction to the fact that $i_{r-1}$ and $i_{r}$ are in the same row.
	\end{proof}
	\begin{corollary}\label{corollary4.5}
		Let $p_{i_1,i_2\ldots,i_r}$ be a Pl\"{u}cker coordinate dividing a $T$-invariant standard monomial. Let $M$ be a matrix in the Deodhar component $\mathcal{R}^{v_{\underline{m}}}_{w_{r,n}}.$ Then $p_{i_1,i_2,\ldots,i_r}(M)$ is the product of diagonals.
	\end{corollary}
	
	\begin{proof}
		By \cref{lemma5.4}, we have $i_{j-1}$ and $i_{j}$ do not simultaneously lie in the interval $[(j-2)q+m_{j-1}+1,(j-1)q+1]$ for all $2 \leq j \leq r.$ Thus, by \cref{lemma5.3}, we have $p_{i_1,i_2,\ldots,i_r}(M)$ is the product of diagonals.
	\end{proof}
	
	To proceed further we set up the following terminology for a $T$-invariant standard monomial (or the corresponding standard Young tableau) which we use frequently. 
	
	Let $f=p_{\tau_1}\cdots p_{\tau_l}$ be a standard monomial and $\Gamma_{f}$ be the corresponding standard Young tableau. Fix an integer $1 \leq m \leq l.$ Let $\tau_m=(i_1, i_2, \ldots, i_r).$ Then by \cref{corollary4.5}, $p_{\tau_m}(M)$ is the product of $(i_j,j)$-th entries of $M,$ where $1 \leq j \leq r.$  Then by "the restriction of the $j$-th entry of the $m$-th Pl\"{u}cker coordinate to $\mathcal{R}^{v_{\underline{m}}}_{w_{r,n}}$" we mean $"(i_j,j)$-th entry of $\mathcal{R}^{v_{\underline{m}}}_{w_{r,n}}.$"
	
	To state \cref{lemma3.8} and \cref{lemma6.2} first we consider the following monomials in $c_{t,j}$'s.
	\begin{itemize}
		\item[(i)]	\begin{align}
			\bigg(\prod_{l=1}^{m_1}\big(\prod_{t=1}^{l-1}c_{t,1}\big)^r\bigg)\bigg(\prod_{l=m_1+1}^{q+1}\prod_{t=1}^{l-1}c_{t,1}\bigg)\bigg(\prod_{l=m_1+2}^{q+1}\big(\prod_{t=1}^{l-2}c_{t,1}\big)^{r-1}\bigg)
			\label{6.1}
		\end{align} 
		\item[(ii)]	\begin{align}
			\begin{split}
				\bigg(\prod_{l=2}^{m_j}\big(\prod_{t=(j-2)q+m_{j-1}+1}^{(j-1)q}c_{t,2}\prod_{t=(j-1)q+1}^{(j-1)q+l-1}c_{t,1}\big)^{r}\bigg)\bigg(\prod_{l=m_j+1}^{q+1}\big(\prod_{t=(j-2)q+m_{j-1}+1}^{(j-1)q}c_{t,2}\prod_{t=(j-1)q+1}^{(j-1)q+l-1}c_{t,1}\big)^{j}\bigg)\\\bigg(\prod_{l=m_j+2}^{q+1}\big(\prod_{t=(j-2)q+m_{j-1}+1}^{(j-1)q}c_{t,2}\prod_{t=(j-1)q+1}^{(j-1)q+l-2}c_{t,1}\big)^{r-j}\bigg)
			\end{split}
			\label{6.2}
		\end{align}
		\item[(iii)]  \begin{align}
			\prod_{l=2}^{q+1}\big(\prod_{t=(r-2)q+m_{r-1}+1}^{(r-1)q}c_{t,2}\prod_{t=(r-1)q+1}^{(r-1)q+l-1}c_{t,1}\big)^r
			\label{6.3}
		\end{align} 
		\item[(iv)] \begin{align}
			\displaystyle\prod_{j=1}^{r-1}\prod_{m=1}^{r-j}\bigg(\big(\prod_{l=t_{m,j}}^qc_{(j-1)q+l,1}\big)\big(\sum_{z=m_j+1}^{t_{m,j}-1}(\prod_{l=m_j+1}^{z-1}c_{(j-1)q+l,2}\prod_{l=z}^{t_{m,j}-1}c_{(j-1)q+l,1})+ \prod_{l=m_j+1}^{t_{m,j}-1}c_{(j-1)q+l,2}\big)\bigg).
			\label{6.6}
		\end{align} 
	\end{itemize}
	
	\begin{lemma}\label{lemma3.8}
		The common factor of the restriction of all $T$-invariant standard monomials to $\mathcal{R}^{v_{\underline{m}}}_{w_{r,n}}$ is the product of the polynomials appearing in \eqref{6.1}, \eqref{6.2} and \eqref{6.3}.
	\end{lemma}
	
	\begin{proof}
		By \cref{lemma3.4}, we have for every $1 \leq l \leq m_1,$ $E_{m,1}=l$ for all $(l-1)r+1 \leq m \leq lr.$ Hence,  for $1 \leq l \leq m_1$ and $(l-1)r+1 \leq m \leq lr,$ the restriction of the first entry of the $m$-th  Pl\"{u}cker coordinate to $\mathcal{R}^{v_{\underline{m}}}_{w_{r,n}}$ is $\prod_{t=1}^{l-1}c_{t,1}$ (see \eqref{(5.1)}). By \cref{lemma3.4}, we have $E_{(l-1)r+1,1}=l$ for all  $m_1+1 \leq l \leq q+1.$ Hence, for all $m_1+1 \leq l \leq q+1$ the restriction of the first entry of the $(l-1)r+1$-th Pl\"{u}cker coordinate to $\mathcal{R}^{v_{\underline{m}}}_{w_{r,n}}$ is $\prod_{t=1}^{l-1}c_{t,1}$ (see \eqref{(5.1)}). Also, By \cref{corollary3.5}, we have for every $m_1+2 \leq l \leq q+1,$ $E_{m,1}$ is either $l-1$ or $l$ for all $(l-2)r+2 \leq m \leq (l-1)r.$ Thus, for $m_1+2 \leq l \leq q+1$ and $(l-2)r+2 \leq m \leq (l-1)r,$ the restriction of the first entry of the $m$-th Pl\"{u}cker coordinate  to $\mathcal{R}^{v_{\underline{m}}}_{w_{r,n}}$ is divisible by $\prod_{t=1}^{l-2}c_{t,1}$ (see \eqref{(5.1)}). Therefore, the common factor corresponding to the first column is the monomial in \eqref{6.1}.

		Let us consider $2 \leq j \leq r-1.$ By \cref{lemma3.4}, we have for all $2 \leq l \leq m_j,$ $E_{m,j}=(j-1)q+l$ for all $(l-1)r-j+2 \leq m \leq lr-j+1.$ Hence, for $2 \leq l \leq m_j$ and $(l-1)r-j+2 \leq m \leq lr-j+1,$ the restriction of the $j$-th entry of the $m$-th Pl\"{u}cker coordinate to $\mathcal{R}^{v_{\underline{m}}}_{w_{r,n}}$ is $\big(\prod_{t=(j-2)q+m_{j-1}+1}^{(j-1)q}c_{t,2}\big)\big(\prod_{t=(j-1)q+1}^{(j-1)q+l-1}c_{t,1}\big)$ (see \eqref{(5.2)}). By \cref{lemma3.4}, for all $m_j+1 \leq l \leq q+1$ we have $E_{m,j}=(j-1)q+l$ for $(l-1)r-j+2 \leq m \leq (l-1)r+1$. Hence, for $m_j+1 \leq l \leq q+1$ and $(l-1)r-j+2 \leq m \leq (l-1)r+1,$ the restriction of the $j$-th entry of the $m$-th Pl\"{u}cker coordinate to $\mathcal{R}^{v_{\underline{m}}}_{w_{r,n}}$ is $\big(\prod_{t=(j-2)q+m_{j-1}+1}^{(j-1)q}c_{t,2}\big)\big(\prod_{t=(j-1)q+1}^{(j-1)q+l-1}c_{t,1}\big)$ (see \eqref{(5.2)}). Also by \cref{corollary3.5}, we have for $m_j+2 \leq l \leq q+1,$ $E_{m,j}$ is either $(j-1)q+l-1$ or $(j-1)q+l$ for all $(l-2)r+2 \leq m \leq (l-1)r-(j-1)$. Hence, for $m_j+2 \leq l \leq q+1$ and $(l-2)r+2 \leq m \leq (l-1)r-(j-1),$ the restriction of the $j$-th entry of the $m$-th Pl\"{u}cker coordinate to $\mathcal{R}^{v_{\underline{m}}}_{w_{r,n}}$ is divisible by $\big(\prod_{t=(j-2)q+m_{j-1}+1}^{(j-1)q} c_{t,2}\big)
		\big(\prod_{t=(j-1)q+1}^{(j-1)q+l-2}c_{t,1}\big)$ (see \eqref{(5.2)}). Therefore, the common factor corresponding to the $j$-th column is the monomial in \eqref{6.2}. 
		
		By \cref{corollary3.3}, we have for all $2 \leq l \leq q+1$ we have $E_{(l-2)r+m,r}=(r-1)q+l$ for all $2 \leq m \leq r+1.$ Therefore, the common factor corresponding to the $r$-th column is the monomial in \eqref{6.3}.
	\end{proof}
	
	\begin{lemma}\label{lemma6.2}
		The non-common factor of the restriction of a $T$-invariant standard monomial to $\mathcal{R}^{v_{\underline{m}}}_{w_{r,n}}$ is appearing in \eqref{6.6}. 
	\end{lemma}
	
	\begin{proof}   Let $f \in R_1$ be a standard monomial and $\Gamma$ be the standard Young tableau associated to $f.$ Then $\Gamma$ has $n$ rows and $r$ columns. Let $E_{i,j}$  be the $(i,j)$-th entry of the tableau $\Gamma$ and $N_{t,j}$ be the number of boxes in the $j$-th column of $\Gamma$ containing $t.$
		
		Let $j$ be a fixed integer such that $1 \leq j \leq r-1.$ By \cref{lemma3.2}, we have $(j-1)q+m_j+1 \leq E_{m,j+1} \leq jq+1$ for all $1 \leq m \leq r-j.$
		Thus, for any $1 \leq m \leq r-j,$ $E_{m,j+1}=(j-1)q+t_{m,j}$ for some fixed integer $t_{m,j}$ such that $m_j+1 \leq t_{m,j} \leq q+1.$ 
		
		Let $1 \leq b_{1,j} \leq r-j$ be the largest positive integer such that $$E_{1,j+1}=E_{2,j+1}=\cdots=E_{b_{1,j},j+1}. $$  
		Since $E_{m,j+1}=(j-1)q+t_{m,j}$ for all $1 \leq m \leq b_{1,j}$ and $E_{b_{1,j},j+1}=(j-1)q+t_{b_{1,j},j},$ we have $t_{m,j}=t_{b_{1,j},j}$ for all $1 \leq m \leq b_{1,j}.$  
		
		Let $b_{1,j} < b_{2,j} \leq r-j$ be the largest positive integer such that $E_{b_{1,j}+1,j+1}=\cdots=E_{b_{2,j},j+1}.$  Then we have $E_{b_{2,j},j+1}= (j-1)q+t_{b_{2,j},j}$ for some fixed integer $t_{b_{2,j},j}$ such that $t_{b_{1,j},j}+1 \leq t_{b_{2,j},j} \leq q+1.$  Since $E_{m,j+1}=(j-1)q+t_{m,j}$ for all $b_{1,j}+1 \leq m \leq b_{2,j},$ we have $t_{m,j}=t_{b_{2,j},j}$ for all $b_{1,j}+1 \leq m \leq b_{2,j}.$ 
		
		Now proceeding in this way we obtain a sequence of positive integers $1 \leq b_{1,j} < b_{2,j} < \cdots < b_{y_j,j}=r-j$ such that $E_{b_{i,j}+1,j+1}=\cdots=E_{b_{i+1,j},j+1}$ for all $1 \leq i \leq y_j-1.$ Thus, for any $1 \leq i \leq y_j-1,$ we have $E_{b_{i+1,j},j+1}= (j-1)q+t_{b_{i+1,j},j}$ for some fixed integer $t_{b_{i+1,j},j}$ such that $t_{b_{i,j},j}+1 \leq t_{b_{i+1,j},j} \leq q+1.$ Since $E_{m,j+1}=(j-1)q+t_{m,j}$ for all $b_{i,j}+1 \leq m \leq b_{i+1,j},$ we have $t_{m,j}=t_{b_{i+1,j},j}$ for all $b_{i,j}+1 \leq m \leq b_{i+1,j}.$ 
		
		Recall that by \cref{corollary3.5}, for any $m_j+1 \leq l \leq q,$ we have $E_{m,j}$ is either $(j-1)q+l$ or $(j-1)q+l+1$ for all $r(l-1)+2 \leq m \leq rl-j+1.$ 
		
		Therefore, for every $m_j+1 \leq l \leq q,$ we study the interval $[r(l-1)+2, rl-j+1]$ case by case.
		
		Case(i): Assume that $m_j+1 \leq l \leq t_{b_{1,j},j}-1.$ Then we claim that $E_{m,j}=(j-1)q+l$ for all $r(l-1)+2 \leq m \leq rl-j+1.$ Since $E_{1,j+1}=(j-1)q+t_{b_{1,j},j},$ we have $(j-1)q+l$ appears exactly $r$ times in the $j$-th column for all $m_j+1 \leq l \leq t_{b_{1,j},j}-1.$ Further, by \cref{lemma3.4}, for $2 \leq l \leq m_j,$ $E_{m,j} = (j-1)q+l$ for all $r(l-1)-j+2 \leq m \leq rl-j+1.$ Hence, for $m_j+1 \leq l \leq t_{b_{1,j},j}-1,$ we have $E_{m,j}=(j-1)q+l$ for all $r(l-1)-j+2 \leq m \leq rl-j+1.$ Therefore, the claim is proved. 
		
		Case(ii): Assume that $1 \leq i \leq y_j-1$ and $t_{b_{i,j},j}\leq l \leq t_{b_{i+1,j},j}-1.$ Then we claim that 	\begin{equation}E_{m,j}=
			\left\{\begin{array}{lr}
				(j-1)q+l & \text{for $r(l-1)+2 \leq m \leq rl-b_{i,j}-j+1$}\\
				(j-1)q+l+1 & \text{for $rl-b_{i,j}-j+2 \leq m \leq rl-j+1$}.\\
			\end{array}\right.
		\end{equation} 
		
		Subcase (1): Assume that $l=t_{b_{1,j},j}.$ Recall that $E_{m,j+1}=(j-1)q+t_{b_{1,j},j}$ for all $1 \leq m \leq b_{1,j}.$ Hence, $N_{(j-1)q+t_{b_{1,j},j},j+1} = b_{1,j}.$ Therefore, $N_{(j-1)q+t_{b_{1,j},j},j} = r-b_{1,j}.$ Further, since by Case (i), $E_{r(t_{b_{1,j},j}-1)-j+1,j} \leq (j-1)q+t_{b_{1,j},j}-1,$ we have $E_{m,j}=(j-1)q+t_{b_{1,j},j}$ for all  $r(t_{b_{1,j},j}-1)-j+2 \leq m \leq rt_{b_{1,j},j}-b_{1,j}-j+1.$

		Subcase (2): Assume that $t_{b_{i,j},j}+1\leq l \leq t_{b_{i+1,j},j}-1$ for $1 \leq i \leq y_j-1.$ 
		Since $E_{b_{i,j}+1,j+1}=(j-1)q+t_{b_{i+1,j},j}$, we have $\sum_{t=2}^lN_{(j-1)q+t,j+1} \leq b_{i,j}.$  Further, since $E_{b_{i,j},j+1}=(j-1)q+t_{b_{i,j},j},$ we have $\sum_{t=2}^{l-1}N_{(j-1)q+t,j+1} \geq b_{i,j}.$ Now consider $n_1=n_2=b_{i,j}.$ Then by using \cref{lemma3.6}, we have $E_{m,j}=(j-1)q+l$ for all $r(l-1)-b_{i,j}-j+2 \leq m \leq rl-b_{i,j}-j+1.$
		
		Subcase (3):	Assume that $l=t_{b_{i+1,j},j}$ for $1 \leq i \leq y_j-1.$ Since $E_{b_{i+1,j}+1,j+1}=(j-1)q+t_{b_{i+2,j},j}$, we have $\sum_{t=2}^{t_{b_{i+1,j},j}}N_{(j-1)q+t,j+1} \leq b_{i+1,j}.$ Further, since $E_{b_{i,j},j+1}=(j-1)q+t_{b_{i,j},j}$, we have $\sum_{t=2}^{t_{b_{i+1,j},j}-1}N_{(j-1)q+t,j+1} \geq b_{i,j}.$ Now consider $n_1=b_{i,j}$ and $n_2=b_{i+1,j}.$ Then by using \cref{lemma3.6}, we have $E_{m,j}=(j-1)q+t_{b_{i+1,j},j}$ for all  $r(t_{b_{i+1,j},j}-1)-b_{i,j}-j+2 \leq m \leq rt_{b_{i+1,j},j}-b_{i+1,j}-j+1.$

		Case (iii): Assume that $t_{b_{y_{j},j},j} \leq l\leq q.$ Then we claim that $E_{m,j}=(j-1)q+l+1$ for all $r(l-1)+2 \leq m \leq rl-j+1.$ Since $E_{b_{y_j,j},j+1}=(j-1)q+t_{b_{y_j,j},j},$ we have $(j-1)q+l+1$ appears $r$ times in the $j$-th column for all $t_{b_{y_{j},j},j} \leq l\leq q.$ Thus, $E_{m,j}=(j-1)q+l+1$ for all $r(l-1)+2 \leq m \leq rl+1.$ Therefore, the claim is proved.

		By Case (i) we have for $m_j+1 \leq l \leq t_{b_{1,j},j}-1$ and $r(l-1)+2 \leq m \leq rl-j+1,$ the non-common factor restriction of the $j$-th entry of the $m$-th Pl\"{u}cker coordinate to $\mathcal{R}^{v_{\underline{m}}}_{v_{r,n}}$ is $1.$ Assume that $1 \leq i \leq y_j-1.$ Then by Case (ii) we have for $t_{b_{i,j},j} \leq l \leq t_{b_{i+1,j},j}-1$ and $r(l-1)+2 \leq m \leq rl-b_{i,j}-j+1,$ the non-common factor of the restriction of the $j$-th entry of the $m$-th Pl\"{u}cker coordinate to $\mathcal{R}^{v_{\underline{m}}}_{w_{r,n}}$ is $1$ (see \eqref{(5.1)} and \eqref{(5.2)}). Also, for $t_{b_{i,j},j} \leq l \leq t_{b_{i+1,j},j}-1$ and $rl-b_{i,j}-j+2 \leq m \leq rl-j+1,$ the non-common factor of the restriction of the $j$-th entry of the $m$-th Pl\"{u}cker coordinate to $\mathcal{R}^{v_{\underline{m}}}_{v_{r,n}}$ is $c_{(j-1)q+l,1}$ (see \eqref{(5.1)} and \eqref{(5.2)}). By Case (iii) we have for $t_{b_{y_j,j},j} \leq l \leq q$ and $r(l-1)+2 \leq m \leq rl-j+1,$ the non-common factor of the restriction of the $j$-th entry of the $m$-th Pl\"{u}cker coordinate to $\mathcal{R}^{v_{\underline{m}}}_{w_{r,n}}$ is $c_{(j-1)q+l,1}$ (see \eqref{(5.1)} and \eqref{(5.2)}).
		
		Therefore, for $m_j+1 \leq l \leq q$ and $r(l-1)+2 \leq m \leq rl-j+1,$ the product of the non-common factors of the restrictions of the $j$-th entries of the $m$-th Pl\"{u}cker coordinates  to $\mathcal{R}^{v_{\underline{m}}}_{w_{r,n}}$ is 	
		\begin{align}
			\begin{split}
				\prod_{i=1}^{y_j-1}(\prod_{l=t_{b_{i,j},j}}^{t_{b_{i+1,j},j}-1}c^{b_{i,j}}_{(j-1)q+l,1})(\prod_{l=t_{b_{y_j,j},j}}^{q}
				c^{r-j}_{(j-1)q+l,1}) 
				=\prod_{m=1}^{r-j}(\prod_{l=t_{m,j}}^{q}c_{(j-1)q+l,1})
			\end{split}
			\label{(6.4)} (\text{see \cref{lemma6.1} in Appendix}).
		\end{align}
		
		Since $E_{m,j+1}=(j-1)q+t_{m,j}$ for $1 \leq m \leq r-j,$ the product of the restrictions of the $j+1$-th entries of the $m$-th Pl\"{u}cker coordinates to $\mathcal{R}^{v_{\underline{m}}}_{w_{r,n}}$ is 
		\begin{align}
			\prod_{m=1}^{r-j}\bigg(\sum_{z=m_j+1}^{t_{m,j}-1}\big((\prod_{l=m_j+1}^{z-1}c_{(j-1)q+l,2})(\prod_{l=z}^{t_{m,j}-1}c_{(j-1)q+l,1})\big)+\prod_{l=m_j+1}^{t_{m,j}-1}c_{(j-1)q+l,2}\bigg) ~~(\text{see } \eqref{(5.1)} \text{and} \eqref{(5.2)}).
			\label{6.5} 
		\end{align}	
		
		Therefore, the non-common factor of the restriction of $f$ to $\mathcal{R}^{v_{\underline{m}}}_{w_{r,n}}$ is the monomial in \ref{6.6}. \\

	\end{proof}
	
	For $1 \leq j \leq r-1$ and $1 \leq m \leq r-j$, let 
	\begin{center}$\displaystyle X_{(j-1)q+t_{m,j}}=(\prod_{l=t_{m,j}}^qc_{(j-1)q+l,1})(\sum_{z=m_j+1}^{t_{m,j}-1}(\prod_{l=m_j+1}^{z-1}c_{(j-1)q+l,2}\prod_{l=z}^{t_{m,j}-1}c_{(j-1)q+l,1})+ \prod_{l=m_j+1}^{t_{m,j}-1}c_{(j-1)q+l,2}).$
	\end{center}
	For the sequence $\underline{t}_j:= t_{1,j} \leq t_{2,j} \leq \cdots \leq t_{r-j,j},$ let $X_{\underline{t}_j}=\prod_{m=1}^{r-j}X_{(j-1)q+t_{m,j}}.$
	
	On the open cell $\mathcal{R}^{v_{\underline{m}}}_{w_{r,n}},$ $\Gamma$ evaluates to \begin{align}
		F\cdot\prod_{j=1}^{r-1}X_{\underline{t}_j}
		\label{(3.37)}
	\end{align} where $\Gamma$ is as in the proof of \cref{lemma6.2} and $F$ is the common factor as described in \cref{lemma3.8}. Note that $F$ is nowhere vanishing function on $\mathcal{R}^{v_{\underline{m}}}_{w_{r,n}}$ corresponding to the distinguished subexpression of $v_{\underline{m}}$ in $w_{r,n}.$ 
	
	\section{Main theorem} \label{section6}
	In this section we study the GIT quotient of $X$ admitting semistable point with respect to the descent of the line bundle $\mathcal{L},$ where $X:=X^{v_{\underline{m}}}_{w_{r,n}}$ and $\mathcal{L}:=\mathcal{L}(n\omega_r).$ 
	
	Recall that $c_{i,j}$ is the coordinate function on $j$-th appearance of $U^{*}_{-\alpha_i}$ in the Deodhar component $\mathcal{R}^{v_{\underline{m}}}_{w_{r,n}}$ (see \cref{section5}). Then $\{c_{y,1}, c_{z,2} : 1 \leq y \leq rq \text{ and } (j-1)q+m_j+1 \leq z \leq jq \text{ such that } 1 \leq j \leq r-1\}$ are algebraically independent. 
	
	For fixed integer $j$ such that $1 
	\leq j \leq r-1,$ let $R((j-1)q+m_j+1)$ be the $\mathbb{C}$-algebra generated by $c_{y,1}$ such that $(j-1)q+m_j+1 \leq y \leq jq.$ Note that $R((j-1)q+m_j+1)$ is a polynomial algebra.
	
	For fixed integers $j,l,$ where $1 \leq j \leq r-1$ and $m_j+2 \leq l \leq q+1,$ let $R((j-1)q+l)$ be the $\mathbb{C}$-algebra generated by $c_{y,1}, c_{z,2}$ such that $(j-1)q+m_j+1 \leq y \leq jq,$ $(j-1)q+m_j+1 \leq z \leq (j-1)q+l-1.$  
	
	\normalsize{Then we note} that $R((j-1)q+l)$ is a polynomial algebra for all $m_j+1 \leq l \leq q+1.$  
	For $1 \leq j \leq r-1$ and $m_j+1 \leq l \leq q+1$, consider the polynmial 
	\small{\begin{center}$\displaystyle X_{(j-1)q+l}=(\prod_{a=l}^qc_{(j-1)q+a,1})(\sum_{b=m_j+1}^{l-1}(\prod_{a=m_j+1}^{b-1}c_{(j-1)q+a,2}\prod_{a=b}^{l-1}c_{(j-1)q+a,1})+ \prod_{a=m_j+1}^{l-1}c_{(j-1)q+a,2})$
	\end{center}} \normalsize{in} $R((j-1)q+l).$
	
	Observe that for $l=m_j+1,$ $X_{(j-1)q+m_j+1}=(\prod_{a=m_j+1}^qc_{(j-1)q+a,1}).$
	
	For $1 \leq j \leq r-1$ and $m_j+2 \leq l \leq q+1,$ take $$f_{(j,l)}=(\prod_{a=l}^qc_{(j-1)q+a,1})\prod_{a=m_j+1}^{l-2}c_{(j-1)q+a,2}$$ and \small{$$g_{(j,l)}=(\prod_{a=l}^qc_{(j-1)q+a,1})(\sum_{b=m_j+1}^{l-1}(\prod_{a=m_j+1}^{b-1}c_{(j-1)q+a,2}\prod_{a=b}^{l-1}c_{(j-1)q+a,1})) $$} in $R((j-1)q+l-1).$  Then $X_{(j-1)q+l}=f_{(j,l)}c_{(j-1)q+l-1,2}+g_{(j,l)} \in R((j-1)q+l-1)[c_{(j-1)q+l-1,2}].$ \normalsize{Since} $\{c_{z,2}:  (j-1)q+m_j+1 \leq z \leq jq \text{ and }  1 \leq j \leq r-1\}$ are algebraically independent, it follows that $\{X_{(j-1)q+l}: 1 \leq j \leq r-1 \text{ and } m_j+1 \leq l \leq q+1\}$ are algebraically independent. Now consider $\{X_{(j-1)q+l} : m_j+1 \leq l \leq q+1\}$ as homogeneous coordinates of $\mathbb{P}^{q-m_j}.$ 
	
	Let $X^{ss}:=X^{ss}_T(\mathcal{L})$ and $X^{s}:=X^{s}_T(\mathcal{L}).$ Let $R_{k}=H^0(X, \mathcal{L}^{\otimes k})^{T}.$ Let $V=(H^{0}(X,\mathcal{L})^T)^{*}.$ Now since the homogeneous coordinate ring of $T \backslash \backslash X^{ss}$ is generated by $R_1$ (see \cref{lemma4.1}), for any $x \in X^{ss}$ there exists a $\Gamma$ in $\mathcal{A}$ such that $p_{\Gamma}(x) \neq 0,$  where  $\mathcal{A}$ is as in Subsection \eqref{section7}. Consider the map \begin{align*}
		\phi: X^{ss} \longrightarrow \mathbb{P}(V)
	\end{align*} defined by $\phi(x)=(p_{\Gamma}(x))_{\Gamma\in \mathcal{A}}.$ By \eqref{(3.37)}, we have \begin{align}
		p_{\Gamma}|_{\mathcal{R}^{v_{\underline{m}}}_{w_{r,n}}}(u)= F\cdot\prod_{j=1}^{r-1}X_{\underline{t}_j}, u\in \mathcal{R}^{v_{\underline{m}}}_{w_{r,n}}
		\label{5.1}
	\end{align} where $X_{\underline{t}_j}=\prod_{m=1}^{r-j}X_{(j-1)q+t_{m,j}},$  $F$ is a nowhere vanishing function on $\mathcal{R}^{v_{\underline{m}}}_{w_{r,n}}$ and $F$ is independent of $\Gamma.$
	
	Recall that $A_i, 1 \leq i \leq r-1$ is as in Subsection \eqref{section7}. Let $Z=(\mathbb{P}^{q-m_1}, \mathcal{O}(r-1)) \times \cdots \times (\mathbb{P}^{q-m_{r-1}}, \mathcal{O}(1)).$ Since dim($V$)=$|\mathcal{A}|= |A_1 \times A_2 \times \cdots \times A_{r-1}|$ (see \cref{lemma4.1*}), we can embed 
	\begin{center}
		$\psi: Z \hookrightarrow \mathbb{P}(V) \hspace{5cm}$
		$(z_{\underline{t}_1},z_{\underline{t}_2}, \ldots, z_{\underline{t}_{r-1}}) \mapsto (\prod_{j=1}^{r-1}z_{\underline{t}_j})_{(\underline{t}_1,\underline{t}_2, \ldots, \underline{t}_{r-1}) \in A_1 \times A_2 \times \cdots \times A_{r-1}}.$
		\hspace{2cm}{(5.2)}
	\end{center}
	By \eqref{5.1} and (5.2) we have $\phi(X^{ss} \cap \mathcal{R}^{v_{\underline{m}}}_{w_{r,n}}) \subseteq \psi\big(Z\big).$ Since $X^{ss} \cap \mathcal{R}^{v_{\underline{m}}}_{w_{r,n}}$ is a $T$-stable dense open subset of $X^{ss},$ we have the image $\phi(X^{ss} \cap \mathcal{R}^{v_{\underline{m}}}_{w_{r,n}})$ is a dense open subset of $\phi(X^{ss}).$ Therefore, $\phi(X^{ss}) \subseteq \psi\big(Z\big).$ On the other hand, by the definition $\phi(X^{ss})=Proj(\oplus_{k \in \mathbb{Z}_{\geq 0}}R_{k})=T \backslash \backslash X^{ss}.$ Further, since $X^{ss}=X^{s}$ (see \cite[Theorem 3.3, p.5]{K1}), we have $\text{dim}(T \backslash \backslash X^{ss})= \text{dim}(X)- \text{dim}(T)=(q-m_1)+(q-m_2)+\cdots+(q-m_{r-1})=\text{dim}(\psi(Z)).$ Hence, we have $\phi(X^{ss})=\psi(Z).$ Thus, $T \backslash \backslash X^{ss}$ is isomorphic to $(\mathbb{P}^{q-m_1}, \mathcal{O}(r-1)) \times \cdots \times (\mathbb{P}^{q-m_{r-1}}, \mathcal{O}(1)).$

	\begin{theorem}\label{theorem8.1}
		Let $\mathcal{M}$ be the descent of $\mathcal{L}$ to $T \backslash \backslash X^{ss}.$ The polarized variety $(T\backslash \backslash X^{ss}, \mathcal{M})$ is isomorphic to $({\mathbb P}^{q-m_1} \times {\mathbb P}^{q-m_2} \times \cdots \times {\mathbb P}^{q-m_{r-1}}, \mathcal{O}_{\mathbb{P}^{q-m_1}}(r-1) \boxtimes \mathcal{O}_{\mathbb{P}^{q-m_2}}(r-2) \boxtimes \cdots \boxtimes \mathcal{O}_{\mathbb{P}^{q-m_{r-1}}}(1)).$ Further, ${\mathbb P}^{q-m_1} \times {\mathbb P}^{q-m_2} \times \cdots \times {\mathbb P}^{q-m_{r-1}}$ is embedded in a projective space via the Segre embedding.
	\end{theorem}
	\begin{proof}
		Follows from the earlier discussion.
	\end{proof}
	\begin{remark}
		
		\begin{itemize}
			\item[(i)] Let $\pi: SL(n, \mathbb{C}) \longrightarrow PSL(n,\mathbb{C})$ be the natural map. Let $T_{ad}=\pi(T).$ By \cite[Example 3.3, p.194]{K3}, the stabilizer of any $x \in (G_{r,n})^{s}_T(\mathcal{L}(n\omega_r))$ in $T_{ad}$ is trivial. Hence, by \cite[Theorem 2.2]{Sko2}, the automorphism group of $T \backslash \backslash (G_{r,n})^{ss}_T(\mathcal{L}(n\omega_{r}))$ is isomorphic to the semi direct product $W \rtimes Aut(S, \alpha_r)$ where $Aut(S, \alpha_r)$ is the subgroup of the automorphism group of the Dynkin diagram of $SL_{n}$ with respect to $T$ and $B$ fixing $\alpha_r.$ Hence,  the automorphism group of $T \backslash \backslash (G_{r,n})^{ss}_T(\mathcal{L}(n\omega_{r}))$ is a finite group.  Therefore, $T \backslash \backslash (G_{r,n})^{ss}_T(\mathcal{L}(n\omega_{r}))$ is not a product of projective spaces.
			
			\item[(ii)] By \cite[Theorem 5.14, p.906]{Bakshi}, the GIT quotient  $T \backslash \backslash (X^{id}_{w_{3,7}})^{ss}_T(\mathcal{L}(7\omega_{3}))$ is a rational normal scroll. Note that for any positive integers $m_1, m_2, \ldots, m_k$ and $a_1, a_2, \ldots, a_k,$ $dim(H^0(\mathbb{P}^{m_1} \times \mathbb{P}^{m_2} \times \cdots \times \mathbb{P}^{m_k}, \mathcal{O}(a_1) \boxtimes \mathcal{O}(a_2) \boxtimes \cdots \boxtimes \mathcal{O}(a_k)))=\prod_{i=1}^kdim(H^{0}(\mathbb{P}^{m_i}, \mathcal{O}(a_i)).$
			 On the other hand, we have $dim(H^{0}(X^{id}_{w_{3,7}}, \mathcal{L}(7\omega_3))^T)=7$ and $7$ is a prime number. Therefore, if   $T \backslash \backslash (X^{id}_{w_{3,7}})^{ss}_T(\mathcal{L}(7\omega_{3}))$ is a product of projective spaces, then $T \backslash \backslash (X^{id}_{w_{3,7}})^{ss}_T(\mathcal{L}(7\omega_{3}))$ is a projective space. Further, since the quotient $T \backslash \backslash (X^{id}_{w_{3,7}})^{ss}_T(\mathcal{L}(7\omega_{3}))$ is a rational normal scroll, it is not a projective space.
			
		\end{itemize}
	\end{remark}
	\begin{corollary}
		Any product of projective spaces $\mathbb{P}^{a_1} \times \mathbb{P}^{a_2} \times \cdots \times \mathbb{P}^{a_l}$ can be written as the GIT quotient of a Richardson variety in some suitable polarized Grassmannian.
	\end{corollary}
	
	\begin{proof}
		Let max$\{a_1,a_2, \ldots, a_l\}=a_m.$ Let $n=(l+1)(a_m+1)+1.$ Take $r=l+1$ and $q=a_m+1.$ Since $0 \leq a_1, a_2, \ldots, a_l \leq q-1,$ we have $1 \leq q-a_1, q-a_2, \ldots, q-a_l \leq q.$ Let $m_i=q-a_i$ for all $1 \leq i \leq l.$ Consider $v_{\underline{m}}= \prod_{j=1}^{l}(s_{(j-1)q+m_j} \cdots s_{j+2}s_{j+1}) \in W^{P^{\alpha_{l+1}}}.$ Then in one line notation $v_{\underline{m}}$ is $(1,m_1+1,q+m_2+1,\ldots,(l-1)q+m_{l}+1).$ By \cref{theorem8.1}, $T \backslash \backslash (X^{v_{\underline{m}}}_{w_{l+1,n}})^{ss}_T(\mathcal{L}(n\omega_{l+1}))$ is isomorphic to  $(\mathbb{P}^{a_1}, \mathcal{O}(l)) \times (\mathbb{P}^{a_2}, \mathcal{O}(l-1)) \times \cdots \times (\mathbb{P}^{a_l}, \mathcal{O}(1)).$
	\end{proof}
	
	\begin{example}
		
		Let $(n,r,q)=(10,3,3).$ Then $w_{3,10}=(s_3s_2s_1)(s_6s_5s_4s_3s_2)(s_9s_8s_7s_6s_5s_4s_3).$ For $1 \leq m_1, m_2 \leq 3$, consider $v_{(m_1,m_2)} = (s_{m_1}s_{m_1-1}\cdots s_2)(s_{3+m_2}s_{3+m_2-1}\cdots s_{3})$.
		Then $w_{3,10}=c.u_{(m_1,m_2)}v_{(m_1,m_2)},$ where $c=(s_3s_2s_1)(s_6s_5s_4)(s_9s_8s_7)$ and $u_{(m_1,m_2)}=(s_3 \ldots s_{m_1+1})(s_{6} \cdots s_{3+m_2+1}).$  Then we have the following subdiagram of the Bruhat lattice $W^{P^{\alpha_{3}}}:$

		\begin{center}
			\begin{tikzpicture}[scale=.7]
				\node (e) at (0,0)  {$v_{(3,3)}$};
				\node (z) at (-1.4,-1) {$s_3$};
				\node (y) at (1.3,-1) {$s_6$};
				\node (g) at (-2,-2) {$v_{(2,3)}$};
				\node (b) at (2,-2) {$v_{(3,2)}$};
				\node (x) at (-3.4,-3) {$s_2$};
				\node (w) at (3.4,-3) {$s_5$};
				\node (v) at (-1.4,-3.2) {$s_6$};
				\node (u) at (1.4,-3.2) {$s_3$};
				\node (h) at (-4,-4) {$v_{(1,3)}$};
				\node (d) at (0,-4) {$v_{(2,2)}$};
				\node (1) at (4,-4) {$v_{(3,1)}$};
				\node (t) at (-3.4,-5.2) {$s_6$};
				\node (s) at (3.4,-5.2) {$s_3$};
				\node (r) at (-1.4,-4.9) {$s_2$};
				\node (q) at (1.4,-4.9) {$s_5$};
				\node (f) at (-2,-6) {$v_{(1,2)}$};
				\node (a) at (2,-6) {$v_{(2,1)}$};
				\node (o) at (-1.4,-7.1) {$s_5$};
				\node (p) at (1.4,-7.1) {$s_2$};
				\node (c) at (0,-8) {$v_{(1,1)}$};
				\node (z) at (0,-10) {Figure $1$};
				\draw (e) -- (g) -- (h) -- (f) -- (c) -- (a) -- (1) -- (b) -- (e);
				\draw (g) -- (d) -- (f);
				\draw (b) -- (d) -- (a);
			\end{tikzpicture}
		\end{center}
		Then we have the following:
		\begin{itemize}
			\item[(1)]	The GIT quotient $T\backslash \backslash (X^{v_{(3,3)}}_{w_{3,10}})^{ss}_T(\mathcal{L}(10\omega_3))$ is isomorphic to a point.
			\item[(2)] The GIT quotient $T\backslash \backslash (X^{v_{(2,3)}}_{w_{3,10}})^{ss}_T(\mathcal{L}(10\omega_3))$ is isomorphic to ${\mathbb P}^{1}$ and the GIT quotient is embedded via the very ample line bundle $\mathcal{O}_{\mathbb{P}^{1}}(2)$. 
			\item[(3)] The GIT quotient $T\backslash \backslash (X^{v_{(3,2)}}_{w_{3,10}})^{ss}_T(\mathcal{L}(10\omega_3))$ is isomorphic to ${\mathbb P}^{1}$ and the GIT quotient is embedded via the very ample line bundle $\mathcal{O}_{\mathbb{P}^{1}}(1)$. 
			\item[(4)] The GIT quotient $T\backslash \backslash (X^{v_{(1,3)}}_{w_{3,10}})^{ss}_T(\mathcal{L}(10\omega_3))$ is isomorphic to ${\mathbb P}^{2}$ and the GIT quotient is embedded via the very ample line bundle $\mathcal{O}_{\mathbb{P}^{2}}(2)$. 
			\item[(5)] The GIT quotient $T\backslash \backslash (X^{v_{(2,2)}}_{w_{3,10}})^{ss}_T(\mathcal{L}(10\omega_3))$ is isomorphic to ${\mathbb P}^{1} \times {\mathbb P}^{1}$ and the GIT quotient is embedded via the very ample line bundle $\mathcal{O}_{\mathbb{P}^{1}}(2) \boxtimes \mathcal{O}_{\mathbb{P}^{1}}(1)$. 
			\item[(6)]The GIT quotient $T\backslash \backslash (X^{v_{(3,1)}}_{w_{3,10}})^{ss}_T(\mathcal{L}(10\omega_3))$ is isomorphic to ${\mathbb P}^{2}$ and the GIT quotient is embedded via the very ample line bundle $\mathcal{O}_{\mathbb{P}^{2}}(1)$. 
			\item[(7)]The GIT quotient $T\backslash \backslash (X^{v_{(1,2)}}_{w_{3,10}})^{ss}_T(\mathcal{L}(10\omega_3))$ is isomorphic to ${\mathbb P}^{2} \times {\mathbb P}^{1}$ and the GIT quotient is embedded via the very ample line bundle $\mathcal{O}_{\mathbb{P}^{2}}(2) \boxtimes \mathcal{O}_{\mathbb{P}^{1}}(1)$.
			\item[(8)] The GIT quotient $T\backslash \backslash (X^{v_{(2,1)}}_{w_{3,10}})^{ss}_T(\mathcal{L}(10\omega_3))$ is isomorphic to ${\mathbb P}^{1} \times {\mathbb P}^{2}$ and the GIT quotient is embedded via the very ample line bundle $\mathcal{O}_{\mathbb{P}^{1}}(2) \boxtimes \mathcal{O}_{\mathbb{P}^{2}}(1)$.
			\item[(9)] The GIT quotient $T\backslash \backslash (X^{v_{(1,1)}}_{w_{3,10}})^{ss}_T(\mathcal{L}(10\omega_3))$ is isomorphic to ${\mathbb P}^{2} \times {\mathbb P}^{2}$ and the GIT quotient is embedded via the very ample line bundle $\mathcal{O}_{\mathbb{P}^{2}}(2) \boxtimes \mathcal{O}_{\mathbb{P}^{2}}(1)$.   
		\end{itemize}
		
	\end{example}
	
	\section{Appendix} In this short appendix we give a proof of the following lemma which we have used earlier in \cref{lemma6.2} (see \eqref{(6.4)}).
	
	\begin{lemma}\label{lemma6.1} For $1 \leq j \leq r-1,$ let $y_j, b_{i,j}, t_{b_{i,j}}$ be as in \cref{lemma6.2}. Then we have \begin{center}
			$\prod_{i=1}^{y_j-1}(\prod_{l=t_{b_{i,j},j}}^{t_{b_{i+1,j},j}-1}c^{b_{i,j}}_{(j-1)q+l,1})(\prod_{l=t_{b_{y_j,j},j}}^{q}
			c^{r-j}_{(j-1)q+l,1}) 
			=\prod_{m=1}^{r-j}(\prod_{l=t_{m,j}}^{q}c_{(j-1)q+l,1}). \hspace{.2cm}(6.1)$
		\end{center}
	\end{lemma}
	\begin{proof}
		\normalsize{} Recall that for $1 \leq i \leq y_j-1,$ $t_{m,j}=t_{b_{i,j},j}$ for all $b_{i-1,j}+1 \leq m \leq b_{i,j}$ (see proof of \cref{lemma6.2}). Define $b_{0,j}=0.$ Then we have  \small{$$(\prod_{l=t_{b_{i,j},j}}^{t_{b_{i+1,j},j}-1}c^{b_{i,j}}_{(j-1)q+l,1})=\prod_{m=b_{i-1,j}+1}^{b_{i,j}}(\prod_{l=t_{m,j}}^{t_{b_{i+1,j},j}-1}c_{(j-1)q+l,1})(\prod_{l=t_{b_{i,j},j}}^{t_{b_{i+1,j},j}-1}c^{b_{i-1,j}}_{(j-1)q+l,1}).$$}
		
		\normalsize{Recall} that  $t_{m,j}=t_{b_{y_j,j},j}$ for all $b_{y_j-1,j}+1 \leq m \leq b_{y_j,j}$ (see proof of \cref{lemma6.2}). Hence, \small{$$(\prod_{l=t_{b_{y_j,j},j}}^{q}c^{b_{y_j,j}}_{(j-1)q+l,1})=\prod_{m=b_{y_j-1,j}+1}^{b_{y_j,j}}(\prod_{l=t_{m,j}}^{q}c_{(j-1)q+l,1})(\prod_{l=t_{b_{y_j,j},j}}^{q}c^{b_{y_j-1,j}}_{(j-1)q+l,1}).$$}
		
		\normalsize{Then} the left hand side of $(6.1)$ becomes 
		\tiny{\begin{align*}
				\begin{split}
					\prod_{i=1}^{y_j-1}\bigg(\prod_{m=b_{i-1,j}+1}^{b_{i,j}}(\prod_{l=t_{m,j}}^{t_{b_{i+1,j},j}-1}c_{(j-1)q+l,1})(\prod_{l=t_{b_{i,j},j}}^{t_{b_{i+1,j},j}-1}c^{b_{i-1,j}}_{(j-1)q+l,1})\bigg) \bigg(\prod_{m=b_{y_j-1,j}+1}^{b_{y_j,j}}(\prod_{l=t_{m,j}}^{q}c_{(j-1)q+l,1})(\prod_{l=t_{b_{y_j,j},j}}^{q}c^{b_{y_j-1,j}}_{(j-1)q+l,1})\bigg).\\
					\hspace{4cm}\normalsize{(6.2)}
				\end{split}
		\end{align*}}
		
		\normalsize{For} any increasing sequence $b_{1,j} < b_{2,j} < \cdots < b_{y_j,j}$  we claim that  \small{\begin{align*}
				\prod_{i=2}^{y_j-1}(\prod_{l=t_{b_{i,j},j}}^{t_{b_{i+1,j},j}-1}c^{b_{i-1,j}}_{(j-1)q+l,1})(\prod_{l=t_{b_{y_j,j},j}}^{q}c^{b_{y_j-1,j}}_{(j-1)q+l,1})=
				\prod_{i=2}^{y_j}(\prod_{l=t_{b_{i,j},j}}^{q}c^{b_{i-1,j}-b_{i-2,j}}_{(j-1)q+l,1}). \hspace{1cm}(6.3)
		\end{align*}}
		
		\normalsize{We} prove by the induction on the length of the sequence. 
		
		\normalsize{Note} that \small{$\displaystyle	\prod_{i=2}^{y_j-1}(\prod_{l=t_{b_{i,j},j}}^{t_{b_{i+1,j},j}-1}c^{b_{i-1,j}}_{(j-1)q+l,1})(\prod_{l=t_{b_{y_j,j},j}}^{q}c^{b_{y_j-1,j}}_{(j-1)q+l,1})$
			
			$\displaystyle =\big(\prod_{i=2}^{y_j-1}(\prod_{l=t_{b_{i,j},j}}^{t_{b_{i+1,j},j}-1}c^{b_{1,j}}_{(j-1)q+l,1})(\prod_{l=t_{b_{y_j,j},j}}^{q}c^{b_{1,j}}_{(j-1)q+l,1})\big)\big(\prod_{i=3}^{y_j-1}(\prod_{l=t_{b_{i,j},j}}^{t_{b_{i+1,j},j}-1}c^{b_{i-1,j}-b_{1,j}}_{(j-1)q+l,1})(\prod_{l=t_{b_{y_j,j},j}}^{q}c^{b_{y_j-1,j}-b_{1,j}}_{(j-1)q+l,1})\big)$
			
			$ =\big(\displaystyle\prod_{l=t_{b_{2,j},j}}^{q}c^{b_{1,j}}_{(j-1)q+l,1}\big)\big(\prod_{i=3}^{y_j-1}(\prod_{l=t_{b_{i,j},j}}^{t_{b_{i+1,j},j}-1}c^{b_{i-1,j}-b_{1,j}}_{(j-1)q+l,1})(\prod_{l=t_{b_{y_j,j},j}}^{q}c^{b_{y_j-1,j}-b_{1,j}}_{(j-1)q+l,1})\big).$}
		
		\normalsize{Since} $b_{2,j}-b_{1,j} < \cdots < b_{y_j,j}-b_{1,j},$ by induction on the length of the sequence, we have 
		
		\small{$\displaystyle\prod_{i=3}^{y_j-1}(\prod_{l=t_{b_{i,j},j}}^{t_{b_{i+1,j},j}-1}c^{b_{i-1,j}-b_{1,j}}_{(j-1)q+l,1})(\prod_{l=t_{b_{y_j,j},j}}^{q}c^{b_{y_j-1,j}-b_{1,j}}_{(j-1)q+l,1})=\prod_{i=3}^{y_j}(\prod_{l=t_{b_{i,j},j}}^{q}c^{b_{i-1,j}-b_{i-2,j}}_{(j-1)q+l,1}).$}
		
		\normalsize{Therefore,} {(6.2)} becomes \small{\begin{align*}
				\prod_{i=1}^{y_j-1}\bigg(\prod_{m=b_{i-1,j}+1}^{b_{i,j}}(\prod_{l=t_{m,j}}^{t_{b_{i+1,j},j}-1}c_{(j-1)q+l,1})\bigg) \bigg(\prod_{m=b_{y_j-1,j}+1}^{b_{y_j,j}}(\prod_{l=t_{m,j}}^{q}c_{(j-1)q+l,1})\bigg)\bigg(\prod_{i=2}^{y_j}(\prod_{l=t_{b_{i,j},j}}^{q}c^{b_{i-1,j}-b_{i-2,j}}_{(j-1)q+l,1})\bigg)
			\end{align*}
			\begin{align*}
				=\prod_{i=1}^{y_j-1}\bigg(\prod_{m=b_{i-1,j}+1}^{b_{i,j}}(\prod_{l=t_{m,j}}^{t_{b_{i+1,j},j}-1}c_{(j-1)q+l,1})\prod_{l=t_{b_{i+1,j},j}}^{q}c^{b_{i,j}-b_{i-1,j}}_{(j-1)q+l,1}\bigg) \bigg(\prod_{m=b_{y_j-1,j}+1}^{b_{y_j,j}}(\prod_{l=t_{m,j}}^{q}c_{(j-1)q+l,1})\bigg)
			\end{align*}
			\begin{align*}
				=\prod_{i=1}^{y_j-1}\bigg(\prod_{m=b_{i-1,j}+1}^{b_{i,j}}(\prod_{l=t_{m,j}}^{q}c_{(j-1)q+l,1})\bigg) \bigg(\prod_{m=b_{y_j-1,j}+1}^{b_{y_j,j}}(\prod_{l=t_{m,j}}^{q}c_{(j-1)q+l,1})\bigg)
				=\prod_{m=1}^{r-j}\bigg(\prod_{l=t_{m,j}}^{q}c_{(j-1)q+l,1}\bigg).
		\end{align*}}
	\end{proof}
	
	{\bf Acknowledgement.} The authors would like to thank the Infosys Foundation for the partial financial support. The first named author also would like to thank MATRICS grant for the partial financial support. This work was done during the second named author's stay at Chennai Mathematical Institute. The second named author would
	like to thank Chennai Mathematical Institute for the hospitality during her stay.

	\normalsize{}	
	

\begin{thebibliography}{222}
		
		\bibitem{Bakshi} S. Bakshi, S. S. Kannan, K. V. Subrahmanyam, {\em Torus quotients of Richardson varieties in the Grassmannian,} Communications in Algebra, vol. 48, issue 2, p. 891--914, 2019.
		
		\bibitem{BL} M. Brion, V. Lakshmibai, {\em A geometric approach to Standard Monomial Theory,} Representation Theory of
		the American Mathematical Society, vol. 7, no. 25, p. 651--680, 2003.
		
		\bibitem{carter} R. W. Carter, Finite groups of Lie Type: Conjugacy Classes and Complex Characters
		(John Wiley and Sons, Inc., New York, 1985).
		
		\bibitem{deodhar} V. V. Deodhar, {\em On some geometric aspects of bruhat orderings. i. a finer decomposition
			of bruhat cells,} Inventiones mathematicae, vol. 79, no. 3, p. 499--511, 1985.
		
		\bibitem{R} R. Hartshorne, {\em Algebraic Geometry,} Graduate Texts in Mathematics book series(GTM, volume 52).
		
		\bibitem{HK} C. Hausmann and A. Knutson, {\em Polygon spaces and Grassmannians,} L’Enseignement Mathematique 43, no. 1-2, p. 173--198, 1997.
		
		\bibitem{How} B. J. Howard, {\em Matroids and geometric invariant theory of torus actions on flag spaces,} J. Algebra 312, no. 1, p. 527--541, 2007.
		
		\bibitem{Hum1} J. E. Humphreys, {\em Introduction to lie algebras and representation theory,} vol. 9, Springer Science $\&$ Business Media, 2012.
		
		\bibitem{Hum2} J. E. Humphreys, {\em Linear algebraic groups,} vol. 21, Springer Science $\&$ Business Media, 2012.
		
		\bibitem{K1} S. S. Kannan, {\em Torus quotients of homogeneous spaces,} Proceeding Mathematical Sciences, vol. 2, 1998.
		
		\bibitem{K2} S. S. Kannan, {\em Torus quotients of homogeneous spaces II,} Proceeding Mathematical Sciences, vol. 109,no. 1, p. 23--39, 1999.
		
		\bibitem{K3} S. S. Kannan, {\em GIT related problems of the flag variety for the action of a maximal torus,} In Groups of Exceptional Type, Coxeter Groups and Related Geometries, p. 189--203. Springer, 2014.
		
		\bibitem{kannanetal}S. S. Kannan, K. Paramasamy, S. K. Pattanayak and S. Upadhyay, {\em Torus quotients
			of Richardson varieties}, Comm. Algebra 46(1), p. 254--261, 2018.
		
		\bibitem{KS} S. S. Kannan, P. Sardar, {\em Torus quotients of homogeneous spaces of the general linear group and the standard representation of certain symmetric groups}, Proc. Indian Acad. Sci. (Math. Sci.) 119(1), p. 81--100, 2009.
		
		\bibitem{knop} F. Knop, H. Kraft, D. Luna and T. Vust, {\em Local Properties of Algebraic Group Actions,} in {\em Algebraic transformation groups and invariant theory}, DMV Seminar, B. 13, Birkh\"{a}user, p. 77--87, 1989.
		
		\bibitem{kodama} Y. Kodama and L. Williams, {\em The deodhar decomposition of the grassmannian and the
			regularity of kp solitons,} Advances in Mathematics, vol. 244, p. 979--1032, 2013.
		
		\bibitem{Kum} S. Kumar, {\em Descent of line bundles to GIT quotients of flag varieties by maximal torus,} Transformation groups, vol.13, no. 3-4, p. 757--771, 2008.
		
		\bibitem{LB} V. Lakshmibai, J. Brown, {\em Flag Varieties. An Interplay of Geometry, Combinatorics, and Representation Theory,} vol. 53, Hindustan Book Agency, New Delhi, 2009.
		
		
		
		
		\bibitem{LS} V. Lakshmibai, K. N. Raghavan, {\em Standard Monomial Theory Invariant Theoretic Approach, Encyclopaedia of Mathematical Sciences}, volume 137, 2008.
		
		\bibitem{littelmann} P. Littelmann, {\em A generalization of the Littlewood-Richardson rule,} J. Algebra (2), p. 328--368, 1990.
		
		\bibitem{marsh} R. Marsh and K. Rietsch, {\em parameterizations of flag varieties,} Representation Theory of the American Mathematical Society, vol. 8, no. 9, p. 212--242, 2004.
		
		\bibitem{Mumford} D. Mumford, J. Fogarty and F. Kirwan, {\em Geometric Invariant Theory,} Third Edition, Springer-Verlag, Berlin, Heidelberg, New York, 1994.
		
		
		\bibitem{NP} A. Nayek, S. K. Pattanayak and S. Jindal, {\em Projective normality of torus quotients of flag varieties}, Journal of Pure and Applied Algebra, vol. 224, issue 10, 2020.
		
		\bibitem{New} P. E. Newstead, {\em Lecture on Introduction to Moduli Problem And Orbit Spaces,} Tata Institute of Fundamental Research, Bombay, 1978.
		
		\bibitem{Sko} A. N. Skorobogatov, {\em On swinnerton-dyer,} in  Annales de la facult\'e des science de Toulouse, vol. 2, 1993. 
		
		\bibitem{Sko2} A. N. Skorobogatov, {\em Automorphisms and forms of toric quotients of homogeneous spaces}
		(Russian). Mat. Sb. 200(10), 107--122 (2009); translation in Sb. Math. 200(9--10), 1521--1536
		(2009).
		
	\end{thebibliography}
\end{document}